\documentclass[11pt]{article}
\usepackage{amsfonts}

\usepackage{graphics}

\usepackage{indentfirst}
\usepackage{cite}
\usepackage{latexsym}
\usepackage{amsmath}
\usepackage{amssymb}
\usepackage[dvips]{epsfig}
\usepackage{amscd}

\def\on{\bar\rho}
\newtheorem{theorem}{Theorem}[section]
\newtheorem{remark}{Remark}[section]

\newtheorem{definition}{Definition}[section]
\newtheorem{lemma}[theorem]{Lemma}

\newtheorem{proposition}[theorem]{Proposition}

\newcommand{\n}{\rho}

\renewcommand{\div}{ {\rm div }  }

\newcommand{\pa}{\partial}
\renewcommand{\r}{\mathbb{R}}

\newcommand{\ia}{\int_0^T}

\newcommand{\bt}{\begin{theorem}}
\newcommand{\bl}{\begin{lemma}}
\newcommand{\el}{\end{lemma}}
\newcommand{\et}{\end{theorem}}
\newcommand{\ga}{\gamma}

\newcommand{\curl}{{\rm curl} }

\newcommand{\de}{\delta}
\newcommand{\ve}{\varepsilon}
\newcommand{\la}{\label}

\newcommand{\ol}{\overline}

\newcommand{\bn}{\begin{eqnarray}}
\newcommand{\en}{\end{eqnarray}}
\newcommand{\bnn}{\begin{eqnarray*}}
\newcommand{\enn}{\end{eqnarray*}}

\newcommand{\bnnn}{\begin{eqnarray*}}
\newcommand{\ennn}{\end{eqnarray*}}

\newcommand{\ba}{\begin{aligned}}
\newcommand{\ea}{\end{aligned}}
\newcommand{\be}{\begin{equation}}
\newcommand{\ee}{\end{equation}}
\def\O{{\Omega }}

\def\norm[#1]#2{\|#2\|_{#1}}

\newcommand{\si}{\sigma}

\def\la{\label}

\def\na{\nabla}
\def\on{\bar\n}

\makeatletter      
\@addtoreset{equation}{section}
\makeatother

\makeatletter      
\@addtoreset{equation}{section}
\makeatother       

\title{Global Classical Solutions to the 3D Density-Dependent Viscosity Compressible Navier-Stokes Equations with Navier-Slip Boundary Condition in a Simply Connected Bounded Domain}
\author{
  Yuebo CAO
 \\
   {\normalsize Department of Mathematics, College of Sciences,}\\
{\normalsize Shihezi University,}\\ {\normalsize
Shihezi 832003, P. R. China }}

\date{ }
\begin{document}
\maketitle

 \begin{abstract}
 This paper concerns the global existence for classical solutions problem to the 3D Density-Dependent Viscosity barotropic compressible Navier-Stokes in $\Omega$ with slip boundary condition, where  $\Omega$ is a simply connected bounded $C^{\infty}$ domain in $\r^3$ and its boundary only has a finite number of 2-dimensional connected components. By a series of a priori estimates, we show that the classical solution to the system exists globally in time under the assumption that the initial energy is suitably small. The initial density of  such a classical solution is allowed to have large oscillations and contain vacuum states. We also adopt some new techniques and methods  to obtain necessary a priori estimates, especially the boundary integral terms estimates. This is the first result concerning the global existence of classical solutions to the compressible Navier-Stokes equations with density containing vacuum initially and viscosity coefficients depending on density for general 3D bounded smooth domains.
 \end{abstract}


\section{Introduction}

We consider the following three-dimensional barotropic compressible Navier-Stokes equations with density-dependent viscosities:
   \be \la{a1}  \begin{cases}\n_t+{\rm div} (\n u)=0,\\
 (\n u)_t+{\rm div}(\n u\otimes u)-\mu\Delta u-\nabla((\lambda+\mu) {\rm
div} u) +\na P(\n) =0, \end{cases}\ee
 where $(x,t)\in\Omega\times (0,T]$, $\Omega$ is a domain in $\r^{3}$, $t\ge 0$ is time,   $x$  is the spatial coordinate. $\rho=\rho(x,t)\geq0, u= (u^1(x,t),u^2(x,t),u^3(x,t))$ and $P(\rho)=a\rho^{\gamma}\,(a>0,\,\gamma>1)$
are the unknown fluid density, velocity and pressure, respectively.

The shear viscosity $\mu$ and the bulk one $\lambda$ satisfy the following restrictions:
\be\la{h3} \mu=\text{constant}>0,\quad \lambda(\rho)=b\rho^\beta,\,\,b=\text{constant}>0,\,\,\beta\in[1,\gamma+1],
\ee
For the initial data,  we require that
\be \la{h2} \n(x,0)=\n_0(x), \quad \n u(x,0)=\n_0u_0(x),\quad x\in \Omega.\ee
In addition, the system is solved subject to Navier-slip boundary condition
\be \la{ch1} u\cdot n=0\,\, \text{and} \,\,\,\curl u\times n=0 \,\,\,\text{on} \,\,\,\partial\Omega,\ee
where $n=(n^1(x),n^2(x),n^3(x))$ is the unit outward normal vector of $x\in\partial\Omega$.

When $\lambda$ and $\mu$ are both positive constants, there is a huge literature on the studies of the well-posedness of solutions to the multi-dimensional compressible Navier-Stokes equations. Nash \cite{Na}, Itaya \cite{itaya} and Tani \cite{TA1} established the local well-posedness theory of classical solutions to both initial value and initial-boundary-value problems in the absence of vacuum. The first pioneering well-known theory of the global well-posedness of classical solutions is due to Matsumura and Nishida \cite{M1}. Later, when the initial energy is suitably small, Huang et al.\cite{hlx1} and Li and Xin \cite{lixinzhou} obtained the global well-posedness of classical solutions to the three-dimensional and two-dimensional isentropic compressible Navier-Stokes equations respectively.

For the case that $\lambda$ and $\mu$ are not constants, that is, $\lambda$ and $\mu$ are both the power function of density, there are extensive studies concerning on the one-dimensional isentropic Navier-Stokes equations under certain restrictions of $\lambda$ and $\mu$. For details, the readers can refer to \cite{Bresch1,Bresch2,Ding,JXZ,LXY,YYZ,ZF} and the references therein. For two-dimensional case, Vaigant and Kazhikhov \cite{Vka1} first proposed and investigated the Navier-Stokes equations when $\mu$ is a positive constant and $\lambda=\rho^\beta$. They obtained a unique global strong solution provided that $\beta>3$ and $\Omega$ is bounded. Later, Huang and Li \cite{hlx1} relaxed the condition $\beta>3$ to the one that $\beta>\frac{4}{3}$ and studied the large-time behavior of the solutions. Huang and Li \cite{liliang} considered the local classical solutions to the cauchy problem of the two-dimensional barotropic compressible Navier-Stokes equations with vacuum in weighted spaces.

We can give different boundary conditions for the Navier-Stokes equations when $\Omega$ is a bounded domain in $\r^2$ or $\r^3$. One of the most common boundary conditions called no-slip boundary condition (i.e., $u = 0$ on $\partial \Omega$) was given by G. Stokes in 1845. Another important boundary condition is Navier-type slip boundary condition,
\be \la{Navi}
u \cdot n = 0, \,\,(D(u)\,n+ \vartheta u)\cdot\tau=0 \,\,\,\text{on}\,\,\, \partial\Omega,
\ee
where $D(u) = (\nabla u+(\nabla u)^{tr})/2$ is the shear stress, $\vartheta$ is a scalar friction function, $(D(u)\,n+ \vartheta u)\cdot\tau$ is the projection of tangent plane of $(D(u)\,n+ \vartheta u)$ on $\partial\Omega$. Navier \cite{Nclm1} first proposed this boundary condition which was followed by numerical studies for fluid mechanical problems, such as \cite{Cpfc1, Itt2, se2} and the references therein. the first rigorous analysis of the Navier-Stokes equations with Naiver type
slip boundary condition is due to Solonnikov and \v{S}\v{c}adilov \cite{Sva1}. They studied the stationary linearized Navier-Stokes system under the boundary
condition:
\be \la{Navi0}
u \cdot n = 0,\,\, (D(u)\,n)\cdot\tau=0 \,\,\,\text{on}\,\,\, \partial\Omega.
\ee
 Vaigant $\&$ Kazhikhov \cite{Vka1} established global calssical large solutions of \eqref{a1} with the boundary condition $u\cdot n=0,\,\,\curl u=0$ on $\partial\Omega$ when $\lambda=\rho^\beta,\,\beta>3$ and $\Omega=[0,1]\times [0,1]$. Hoff \cite{Ho3} studied the global existence of weak solution with the Navier-type  slip boundary condition on the half space in $\r^3$. It is worth noting that in \cite{Vka1, Ho3}, the initial density is away from vacuum and the boundary of $\Omega$ is flat. Recently, when $\lambda$ and $\mu$ are both positive constants, Cai-Li \cite{CCL1} obtained the global classical solutions to the 3D compressible Navier-Stokes equations with Navier-type slip boundary condition in general 3D bounded smooth domains, where the initial density may contain vacuum. However, for the Navier-Stokes equations in general 3D bounded domains (the boundary is not necessary flat), when viscosity coefficients satisfy \eqref{h3} and initial density may contain vacuum, it seems to be no relevant results. This is exactly what we want to study in this paper.


\begin{definition}
we say that $\Omega$ is simply connected if the first Betti number of $\Omega$ in $\r^3$ vanishes, that is, any simple closed curve in $\Omega$ can be contracted to a point in $\Omega$. If the second Betti number of $\Omega$ is zero, we say that $\Omega$ has no holes.
\end{definition}

For   integer $k$ and $1\leq q<+\infty$, $W^{k,q}(\Omega)$ is the standard Sobolev spaces and  $$ W_0^{1,q}(\Omega)\triangleq\{u\in W^{1,q}(\Omega)~\text{:}~u~ \text{is equipped with zero trace on } \partial{\Omega}\}.$$

Set
$$ \int f dx\triangleq\int_{\Omega}fdx,$$
and
\bnn \bar{f}\triangleq\frac{1}{|\Omega|}\int_{\Omega}f  dx,\enn
where $|\Omega|$ is the measure of $\Omega.$

For simplicity, we denote $L^q(\Omega)$, $W^{k,q}(\Omega)$, $H^k(\Omega) \triangleq W^{k,2}(\Omega),$ and $H_0^1(\Omega)\triangleq W_0^{1,2}(\Omega)$ by $L^q$,  $W^{k,q}$, $H^k$, and $H^1_0$ respectively.

For two $3\times 3$  matrices $A=\{a_{ij}\},\,\,B=\{b_{ij}\}$, we define
 $$ A\colon  B\triangleq \text{tr} (AB)=\sum\limits_{i,j=3}^{n}a_{ij}b_{ji}.$$


The initial total energy of (\ref{a1})   is defined as
\be \la{c0}
C_0 \triangleq\int_{\Omega}\left(\frac{1}{2}\n_0|u_0|^2 +G(\rho_0) \right)dx.
\ee

 The  main purpose of this paper is to establish the following global existence   of classical solutions of $\eqref{a1}$-$\eqref{ch1}$ in a gneral smooth bounded domain $\Omega\subset\r^3.$
\begin{theorem}\la{th1}  Let $\Omega$ be a simply connected bounded domain in $\r^3$ and its $C^{\infty}$ boundary $\partial\Omega$ has a finite number of 2-dimensional connected components. For some given constant $M>0$ (not necessarily small), suppose that the initial data $(\n_0,u_0)$ satisfy
\be \la{dt1}   u_0\in
H^3,\quad (\rho_0,P(\rho_0))\in  H^{3},\ee
\be\la{dt2} 0\leq\rho_0\leq\hat{\rho},~~\|\nabla u_0\|_{L^2}^2\leq M, \ee
and the compatibility condition
\be\la{dt3}
-\mu\triangle u_0-\nabla((\lambda(\rho_0)+\mu)\div u_0) + \nabla P(\rho_0) = \rho_0g,
\ee
for some  $ g\in L^2.$
Then there exists a positive constant $\ve$ depending only on  $\mu$, $\lambda$, $\ga$, $a$, $\hat{\rho}$, $\Omega$ and $M$  such that
\be\la{dt30}
C_0\leq\varepsilon, \ee
then the system \eqref{a1}-\eqref{ch1} has a unique global classical solution $(\rho,u)$ in $\Omega\times(0,\infty)$ satisfying for any $0<\tau<T<\infty$,
\be\la{dt5}
  0\le\n(x,t)\le 2\hat{\rho},\quad  (x,t)\in \O\times(0,\infty),
\ee
\be\la{dt6}\begin{cases}
(\rho,P)\in C([0,T];H^{3} ),\\u\in C([0,T];H^3 )\cap  L^2(0,T; H^{4})\cap L^\infty(\tau,T;H^4),\\
u_t\in L^{\infty}(0,T; H^1)\cap L^2(0,T; H^2)\cap L^\infty(\tau,T;H^2)\cap H^1(\tau,T;H^1),\\   \sqrt{\rho}u_t\in L^\infty(0,T;L^2).
\end{cases}\ee
 and that for any $0<T<\infty,$
\be\la{dt55}
  \tilde{C}(T)\inf_{x\in\Omega}\rho_0(x)\le\n(x,t)\le 2\hat{\n},\quad  (x,t)\in \O\times[0,T],
\ee
some positive constant $\tilde{C}(T)$ depending only on $T,$ $\mu,$ $\lambda,$ $\ga,$  $a,$ $\hat\rho,$ $\Omega$ and $M$.
Moreover,  for any $r\in [1,\infty)$ and $p\in [1,6],$ there exist positive constants $C$ and $\tilde{\eta}$ depending only  on $\mu,$  $\lambda,$  $\gamma,$ $a$,   $\hat{\rho}$, $M, \bar{\rho}_0$,  $\Omega$, $r$, and $p$   such that
\be  \la{qa1w} \left(\|\rho-\bar{\rho}_0\|_{L^r}+\|  u\|_{W^{1,p}} +\|\sqrt{\rho}\dot{u}\|^2_{L^2}\right)\leq Ce^{-\tilde{\eta} t}.\ee
\end{theorem}

Then, with the   exponential decay rate  \eqref{qa1w} at hand,  modifying slightly the proof of \cite[Theorem 1.2]{lx},
we can   establish  the following large-time behavior of the gradient of the
density when vacuum states appear initially.
\begin{theorem}\la{th2}
Under the conditions of Theorem \ref{th1}, assume further  that
there exists some point $x_0\in \Omega$ such that $\n_0(x_0)=0.$  Then the unique
global classical solution $(\n,u)$ to the   problem  \eqref{a1}-\eqref{ch1} obtained in
Theorem \ref{th1}  satisfies that for any $\tilde{r}>3,$   there exist positive constants $\hat{C}_1$ and $\hat{C}_2$ depending only  on $\mu$,  $\lambda$,  $\gamma$, $a$,   $\hat{\rho}$, $M,$ $\bar{\n}_0$, $\Omega$ and $\tilde{r}$   such that for any $t>0$,
\be\la{qa2w}\ba \|\na\n (\cdot,t)\|_{L^{\tilde{r}}}\geq \hat{C}_1 e^{\hat{C}_2 t} . \ea\ee
\end{theorem}
A few remarks are in order:
\begin{remark} It follows from Sobolev's inequality and \eqref{dt6}$_1$  that \be\la{soh1}  \n,\na \n \in C(\bar\Omega\times [0,T]).\ee
\eqref{dt6}$_2$ and \eqref{dt6}$_3$ show that  \be \la{soh2} u,\na u, \na^2 u, u_t \in C(\bar\Omega\times [\tau,T]),\ee due to the following simple fact: $$L^2(\tau,T;H^1)\cap H^1(\tau,T;H^{-1})\hookrightarrow C([\tau,T];L^2).$$
Finally, by \eqref{a1}$_1,$ we obtain \be \n_t=-u\cdot\na \n-\n\div u\in C(\bar\Omega\times [\tau,T]),\ee which together with \eqref{soh1} and \eqref{soh2} shows that the solution obtained by Theorem \ref{th1} is a classical one.
\end{remark}

\begin{remark}
The conclusion of Theorem \ref{th1} still holds after a little modification of the proof process if the boundary condition \eqref{ch1} is changed to a more general condition
\be \la{ch2} u\cdot n=0\,\, \text{and} \,\,\,\curl u\times n=-A u\,\,\,\text{on} \,\,\,\partial\Omega,\ee
where $n=(n^1(x),n^2(x),n^3(x))$ is the unit outward normal vector of $x\in\partial\Omega$, $ A = A(x)$ is $3\times3$ smooth
and positive semi-definite symmetric matrix defined on $\partial\Omega$.
\end{remark}
\begin{remark} Theorem \ref{th1} is the first result concerning the global existence of the compressible Navier-Stokes equations \eqref{a1} with the density containing vacuum initially and viscosity coefficients satisfying \eqref{h3} for general 3D bounded smooth domains.\end{remark}

We now comment on this paper. Our work is motivated by Li-Xin \cite{jx01}, Li-Liang \cite{liliang} and Cai-Li \cite{CCL1}.
Our research is based on the following three important observations. First, we rewrite $ (\ref{a1})_2 $ as
\be \la{hod1}\ba
\rho\dot{u}=\nabla F - \mu\nabla\times\curl u ,
\ea \ee with
\be \la{dt0} \dot v\triangleq v_t+u\cdot\nabla v,\quad \text{curl} u \triangleq \nabla\times u ,\quad F\triangleq(\lambda(\rho)+2\mu)\,\div u-(P-\Bar P),\ee where $v=(v_1,v_2,v_3)$, $\dot v$ and $F$ are called the
material derivative of $v$ and the effective viscous flux respectively. The second one is that the following inequality
$$\|\nabla u\|_{L^q}\leq C(\|\div u\|_{L^q}+\|\curl u\|_{L^q})\,\,\,\text{for any} \,\,\,q>1,$$ is still hold when $u\in W^{1,q}$ with $u\cdot n=0$ on $\partial\Omega$ and the first Betti number of $\Omega$ vanishes due to \cite{vww}. The last important observation is the following equality:
\be\la{pzw1} u\cdot\nabla u\cdot n=-u\cdot\nabla n\cdot u,\ee
since $u\cdot n=0$ on $\partial\Omega$.

The structure of this paper is as follows. First, we give some known facts and elementary inequalities which will be used in Section 2. In Section 3 and Section 4, some necessary a priori estimates on classical solutions will be obtained. Finally, Theorems \ref{th1} will be proved in Section 5.
\section{Preliminaries}\la{se2}

In this section, we recall some functional spaces, known facts and elementary inequalities which will be used later.

First, similar to the proof of \cite[Theorem 1.4]{hxd1}, we have the local existence of strong and classical solutions.
\begin{lemma}\la{loc1} Let $\Omega$ be as in Theorem \ref{th1}. assume that $(\n_0,u_0)$ satisfies \eqref{dt1} and \eqref{dt3}. Then there exist a small time $T>0$ and a unique classical solution $(\n,u)$ to the problem \eqref{a1}-\eqref{ch1} on $\Omega\times(0,T]$ satisfying for any $ \tau\in(0,T),$
\be\nonumber\begin{cases}
 (\rho,P)\in C([0,T];H^{3} ),\\u\in C([0,T];H^3 )\cap  L^2(0,T; H^{4}),\\
u_t\in L^{\infty}(0,T; H^1)\cap L^2(0,T; H^2),\\
\sqrt{\rho}u_t\in L^\infty(0,T;L^2).\\
\end{cases}\ee \end{lemma}




First, the well-known Gagliardo-Nirenberg inequality (see \cite{nir}) will be used more than once.
\begin{lemma}
[Gagliardo-Nirenberg]\la{l1} Assume that $\Omega$ is a bounded Lipschitz domain in $\r^3$. For  $p\in [2,6],\,q\in(1,\infty), $ and
$ r\in  (3,\infty),$ there exist two generic
 constants
$C_1,\,\,C_2>0$ which may depend  on $p$, $q$, $r$ and $\Omega$ such that for any  $f\in H^1({\O }) $
and $g\in  L^q(\O )\cap W^{1,r}(\O), $
\be\la{g1}\|f\|_{L^p(\O)}\le C_1 \|f\|_{L^2}^{\frac{6-p}{2p}}\|\na
f\|_{L^2}^{\frac{3p-6}{2p}}+C_2\|f\|_{L^2} ,\ee
\be\la{g2}\|g\|_{C\left(\ol{\O }\right)} \le C_1
\|g\|_{L^q}^{q(r-3)/(3r+q(r-3))}\|\na g\|_{L^r}^{3r/(3r+q(r-3))} + C_2\|g\|_{L^2}.
\ee
Moreover, if $f\cdot n|_{\partial\Omega}=0,\,\,\,g\cdot n|_{\partial\Omega}=0$, then the constant $C_2=0.$
\end{lemma}

The following Zlotnik  inequality will be used to get the uniform (in time) upper bound of the density $\n.$
\begin{lemma}[\cite{zl1}]\la{le1}   Suppose the function $y$ satisfy
\bnn y'(t)= g(y)+b'(t) \mbox{  on  } [0,T] ,\quad y(0)=y^0, \enn
with $ g\in C(R)$ and $y, b\in W^{1,1}(0,T).$ If $g(\infty)=-\infty$
and \be\la{a100} b(t_2) -b(t_1) \le N_0 +N_1(t_2-t_1)\ee for all
$0\le t_1<t_2\le T$
  with some $N_0\ge 0$ and $N_1\ge 0,$ then
\bnn y(t)\le \max\left\{y^0,\zeta_0 \right\}+N_0<\infty
\mbox{ on
 } [0,T],
\enn
  where $\zeta_0$ is a constant such
that \be\la{a101} g(\zeta)\le -N_1 \quad\mbox{ for }\quad \zeta\ge \zeta_0.\ee
\end{lemma}





The following two lemmas are given in \cite{vww,CANEHS}.
\begin{lemma}   \la{crle1}
Let $k\geq0$ be a integer, $1<q<+\infty$, and assume that $\Omega$ is a simply connected bounded domain in $\r^3$ with $C^{k+1,1}$ boundary $\partial\Omega$. Then for $v\in W^{k+1,q}$ with $v\cdot n=0$ on $\partial\Omega$, it holds that
$$\|v\|_{W^{k+1,q}}\leq C(\|\div v\|_{W^{k,q}}+\|\curl v\|_{W^{k,q}}).$$
In particular, for $k=0$, we have
$$\|\nabla v\|_{L^q}\leq C(\|\div v\|_{L^q}+\|\curl v\|_{L^q}).$$
\end{lemma}
\begin{lemma}   \la{crle2}
Let $k\geq0$ be a integer, $1<q<+\infty$. Suppose that $\Omega$ is a bounded domain in $\r^3$ and its $C^{k+1,1}$ boundary $\partial\Omega$ only has a finite number of 2-dimensional connected components. Then for $v\in W^{k+1,q}$ with $v\times n=0$ on $\partial\Omega$, we have
$$\|v\|_{W^{k+1,q}}\leq C(\|\div v\|_{W^{k,q}}+\|\curl v\|_{W^{k,q}}+\|v\|_{L^q}).$$
In particular, if  $\Omega$ has no holes, then
$$\|v\|_{W^{k+1,q}}\leq C(\|\div v\|_{W^{k,q}}+\|\curl v\|_{W^{k,q}}).$$
\end{lemma}


\begin{lemma}[\cite{CCL1}]\la{le9}
Let $\Omega$ be a simply connected bounded domain in $\r^3$ with smooth boundary. For $3<q<\infty$, assume that $u\cdot n=0$ and $\curl u\times n=0$ on $\partial\Omega$, $ u\in W^{2,q}$. Then there is a constant  $C=C(q)$ such that  the following estimate holds
\bnn\ba
\|\na u\|_{L^\infty}\le C\left(\|{\rm div}u\|_{L^\infty}+\|\curl u\|_{L^\infty} \right)\ln(e+\|\na^2u\|_{L^q})+C\|\na u\|_{L^2} +C .
\ea\enn
\end{lemma}


\section{\la{se3} A priori estimates (I): lower order estimates}

From now on, we always assume that $\Omega$ is a simply connected bounded domain in $\r^3$ and its $C^{\infty}$ boundary $\partial\Omega$ only has a finite number of 2-dimensional connected components. Suppose $T>0$ be a fixed time and $(\rho,u)$ be a smooth solution to (\ref{a1})-(\ref{ch1})  on
$\Omega \times (0,T]$  with smooth initial
data $(\n_0,u_0)$ satisfying \eqref{dt1} and \eqref{dt2}. We will establish some necessary a priori bounds for smooth solutions to the problem (\ref{a1})-(\ref{ch1}) to extend the local  classical solutions guaranteed by
Lemma \ref{loc1}.

Since $u\cdot n=0$ on $\partial\Omega$, it is easy to check that
\be\la{bdd1}u\cdot\nabla u\cdot n=-u\cdot\nabla n\cdot u \,\,\,\text{on}\,\,\,\partial\Omega.\ee
The domain of definition of the function $n$ can be extended to the interior of $\Omega$, such as $n\in C^3(\bar{\Omega})$. In the following discussion, we still use $n$ to denote the extended function.

The following lemma depends on $u\cdot n=0$ on $\partial\Omega$.
\begin{lemma}[\cite{CCL1}]\la{uup1}If $(\n,u)$ is a smooth solution of
   (\ref{a1}) with slip condition \eqref{ch1}, then for any $p\in[2,6]$, there exists a positive constant $C$ depending only on $p$ and $\Omega$ such that
\be\la{tb90}
\ba\|\dot{u}\|_{L^6}\le C(\|\nabla\dot{u}\|_{L^2}+\|\nabla u\|_{L^2}^2),
\ea\ee
\be\la{tb11}\ba
\|\nabla\dot{u}\|_{L^2}\le C(\|\div \dot{u}\|_{L^2}+\|\curl \dot{u}\|_{L^2}+\|\nabla u\|_{L^4}^2).
\ea\ee
\end{lemma}

In the following, $C$ denotes a generic positive constant depending on $\mu ,  \lambda ,   \ga ,  a ,  \hat{\rho}$ and $\Omega$ and use $C(\alpha)$ to emphasize that $C$ depends on $\alpha.$ The following lemma gives a standard energy estimate for $(\rho,u)$.
\begin{lemma}\la{le2}
 Let $(\n,u)$ be a smooth solution of
 \eqref{a1}--\eqref{ch1} on $\O \times (0,T]. $
  Then there is a positive constant
  $C $ depending only  on $\mu$ and $\Omega$  such that
\be \la{a16} \sup_{0\le t\le T}
\int\left( \frac{1}{2}\rho |u|^2+G(\rho)\right)dx + \int_0^{T}\int((\lambda(\rho)+2\mu)(\div u)^2+\mu|\curl u|^2)dxdt\le C_0,\ee
\be \la{bcgs1}\int_{0}^{T}\|\nabla u\|_{L^{2}}^{2}dt\leq CC_{0}.\ee
\end{lemma}

\begin{proof}
$(\ref{a1})_2 $ can be rewrite as
\be\la{m1} \ba
\rho \dot{u} -\nabla((\lambda(\rho) + 2\mu)\div u)+\mu\nabla\times\curl u + \nabla P  =0,
\ea \ee
since $$-\Delta u=-\nabla\div u+\nabla\times\curl u.$$

Multiply \eqref{m1} by $u$ and integrating over $\Omega$ show


\be\la{mh4} \ba
\frac{1}{2}\left(\int\rho |u|^{2}dx\right)_t  + \int(\lambda(\rho) + 2\mu)(\div u)^{2}dx + \mu\int|\curl u|^{2}dx=\int P\div u dx.
\ea \ee
It is easy to check that
\be\la{Pu1} \ba
 P_t+\div(Pu)+(\gamma-1)P\div u=0,
 \ea \ee
 or
 \be\la{Pu2} \ba
 P_t+\nabla P\cdot u+\gamma P\div u=0,
 \ea \ee
 due to $(\ref{a1})_1$.

After a simple calculation, one has
$$ (G(\rho))_t + \div(G(\rho)u)+(P-P(\bar{\rho})) \div u=0.$$
Integrating over $\Omega$ and using the boundary condition \eqref{ch1}, we get
\be\la{m0} \ba
\left(\int G(\rho)dx\right)_t+\int(P-P(\bar{\rho})) \div udx=0.
\ea \ee
\eqref{m0} combined with \eqref{mh4} gives
\be\la{buchong} \ba
\left(\int\frac{1}{2}\rho |u|^{2}+G(\rho)dx\right)_t  + \int(\lambda(\rho) + 2\mu)(\div u)^{2}dx + \mu\int|\curl u|^{2}dx=0.
\ea \ee
Together with Lemma \ref{crle1}, we give \eqref{bcgs1}.

\end{proof}

Set $\si=\si(t)\triangleq\min\{1,t \},$  define
 \be\la{As1}
  A_1(T) \triangleq \sup_{   0\le t\le T  }\left(\sigma\|\nabla u\|_{L^2}^2\right) + \int_0^{T}\int\sigma
 \n|\dot{u} |^2 dxdt,
  \ee
\be \la{As2}
  A_2(T)  \triangleq\sup_{  0\le t\le T   }\sigma^3\int\n|\dot{u}|^2dx + \int_0^{T}\int
  \sigma^3|\nabla\dot{u}|^2dxdt,
\ee
and
\be \la{As3}
  A_3(T)  \triangleq\sup_{  0\le t\le T   }\|\nabla u\|_{L^2}^2.
\ee

\begin{proposition}\la{pr1}  Under  the conditions of Theorem \ref{th1}, for given number $M>0$ (not necessarily small),
   there exist  positive constants  $\varepsilon$ and $K$ both depending on  $\mu$, $\lambda$, $a$, $\ga$, $\hat{\rho}$, $\Omega$, and $M$  such that if
       $(\rho,u)$  is a smooth solution of
       \eqref{a1}--\eqref{ch1}  on $\Omega\times (0,T] $
        satisfying
 \be\la{zz1}
 \sup\limits_{
 \Omega\times [0,T]}\rho\le 2\hat{\rho},\quad
     A_1(T) + A_2(T) \le 2C_0^{1/3},\quad A_3(\sigma(T))\leq 3K,
  \ee
 then the following estimates hold
        \be\la{zz2}
 \sup\limits_{\Omega\times [0,T]}\rho\le 7\hat{\rho}/4, \quad
     A_1(T) + A_2(T) \le  C_0^{1/3},\quad A_3(\sigma(T))\leq 2K,
  \ee
   provided $C_0\le \ve.$
\end{proposition}
\begin{proof}Proposition \ref{pr1} is a direct consequence of the following Lemmas \ref{nzc1}--\ref{le7}.
\end{proof}

Now, we give the following estimates on $F$, $\curl u$ and $\nabla u$.
\begin{lemma}   \la{le3}
 Assume $\Omega$ is a simply connected bounded domain in $\r^3$ and its $C^{\infty}$ boundary $\partial\Omega$ only has a finite number of 2-dimensional connected components. Let $(\rho,u)$ be a smooth solution of \eqref{a1} satisfying $\sup\limits_{
 \Omega\times [0,T]}\rho\le 2\hat{\rho}$ in $\Omega$ with slip condition \eqref{ch1}. Then for any $p\in[2,6],\,\,1<q<+\infty,$ there exists a positive constant $C$ depending only on $p$, $q$, $\mu$, $\lambda$, $\hat{\rho}$, and $\Omega$ such that
\be\la{tdu1}\ba
\|\nabla u\|_{L^p}\leq C(\|\div u\|_{L^p}+\|\curl u\|_{L^p}),
\ea\ee
\be\la{h19}\ba
\|\nabla F\|_{W^{k+1,q}}\leq C\|\rho\dot{u}\|_{W^{k+1,q}},\,\,k\geq0,
\ea\ee
\be\la{zh19}\ba
\|\nabla\curl u\|_{L^p}\leq C(\|\rho\dot{u}\|_{L^p}+\|\nabla u\|_{L^2}),
\ea\ee
\be\la{h20}\ba
\|F\|_{L^p}+\|\curl u\|_{L^p}&\leq C\|\rho\dot{u}\|_{L^2}^{(3p-6)/(2p)}(\|\nabla u\|_{L^2}+\|P-\bar{P}\|_{L^2})^{(6-p)/(2p)}\\
&\quad+C(\|\nabla u\|_{L^2}+\|P-\bar{P}\|_{L^2}).
\ea\ee
Moreover,
\be\la{hh20}\ba
\|F\|_{L^p}+\|\curl u\|_{L^p}\leq C(\|\rho\dot{u}\|_{L^2}+\|\nabla u\|_{L^2}+\|P-\bar{P}\|_{L^2}),
\ea\ee
\be\la{h18}\ba
\|\nabla u\|_{L^p}&\leq C\|\rho\dot{u}\|_{L^2}^{(3p-6)/(2p)}(\|\nabla u\|_{L^2}+\|P-\bar{P}\|_{L^2})^{(6-p)/(2p)}\\
&\quad+C(\|\nabla u\|_{L^2}+\|P-\bar{P}\|_{L^2}+\|P-\bar{P}\|_{L^p}).
\ea\ee
\end{lemma}
\begin{proof}
We only prove \eqref{h20}. For others, please refer to \cite{CCL1}.

One can deduce from \eqref{g1} and \eqref{h19} that for $p\in[2,6]$,
\be\la{x2610}\ba
\|F\|_{L^p}&\leq C\|F\|_{L^2}^{(6-p)/(2p)}\|\nabla F\|_{L^2}^{(3p-6)/(2p)}+C\|F\|_{L^2}\\
&\leq C\|\rho\dot{u}\|_{L^2}^{(3p-6)/(2p)}(\|\nabla u\|_{L^2}+\|P-\bar{P}\|_{L^2})^{(6-p)/(2p)}\\
&\quad+C(\|\nabla u\|_{L^2}+\|P-\bar{P}\|_{L^2}),
\ea\ee
which also implies that
\be\la{x2611}\ba
\|F\|_{L^p}\leq C(\|\rho\dot{u}\|_{L^2}+\|\nabla u\|_{L^2}+\|P-\bar{P}\|_{L^2}).
\ea\ee

Similarly, \eqref{x2610} and \eqref{x2611} are still true if we replace $F$ by $\curl u$, and so \eqref{h20} and \eqref{hh20} are established.

\end{proof}

Consider the problem
\bn\la{e480}\begin{cases}
{\rm div}\phi=g, \,\,\,\,  &x\in\Omega, \\
\phi=0,\,\,\,&x\in{\partial\Omega},
\end{cases} \en
where $\Omega$ is a bounded domain in $R^{3}$ with Lipschitz boundary.

This problem has been extensively discussed, and due to Theorem 3.3 in \cite{GPG}, one has the following conclusion.
\begin{lemma} \la{ll27}
Suppose $g\in L^{q}(\Omega)$ with $\int gdx=0,\,\,\,q>1$, then the problem \eqref{e480} has a unique solution $\phi\in W_0^{1,q}(\Omega)$, such that
$$\|\phi\|_{W_0^{1,q}(\Omega)}\leq C\|g\|_{L^{q}(\Omega)},$$
where $C$ is a positive constant independent of $g$.

Moreover, if $g={\rm div}h$ for some $h\in L^{r}(\Omega),\,\,\,r>1$ with $h\cdot n=0$ on $\partial\Omega$, then
$$\|\phi\|_{L^{r}(\Omega)}\leq C\|h\|_{L^{r}(\Omega)},$$
where $C$ is a positive constant independent of $h$.
\end{lemma}
The lemma below gives a priori estimate of $\|\rho-\bar{\rho}\|_{L^2(\Omega\times(0,T])}$, which only depends on the initial total energy.
\begin{lemma}   \la{th00}
Suppose~$\sup\limits_{\Omega\times [0,T]}\rho\le 2\hat{\rho}$, then there exists $C=C(\hat{\rho})$ such that
$$\sup_{   0\le t\le T  }\sigma\|P-\bar{P}\|_{L^2}^2+\int_0^T\|P-\bar{P}\|_{L^2}^2dt\leq CC_0^{1/2},\,\,\int_0^T\sigma\|P-\bar{P}\|_{L^2}^2dt\leq CC_0^{3/4}.$$
 \end{lemma}
\begin{proof}  For any $t\in(0,T],\,\,\phi(\cdot,t)$ solves the problem
\be\la{e1}\begin{cases}
\div\phi(\cdot,t)=P-\bar{P}, \,\,\, &x\in\Omega, \\
\phi=0, \,\,\,&x\in{\partial\Omega}.
\end{cases} \ee
By Lemma \ref{ll27},
\be\la{e2}
\|\phi\|_{H_0^{1}}\leq C\|P-\bar{P}\|_{L^{2}},\,\,\|\phi_t\|_{L^{2}}\leq C(\|\rho u\|_{L^{2}}+\|\nabla u\|_{L^{2}})\leq C\|\nabla u\|_{L^{2}}.
\ee

Multiplying $\eqref{a1}_2$ by $\phi$ and integrating over $\Omega\times(0,T],$ one has
\be\la{e3} \ba
\int\rho\dot{u}\cdot\phi dx-\mu\int\Delta u\cdot\phi dx-\int\nabla((\lambda(\rho)+\mu)\div u)\cdot\phi dx+\int\nabla(P-\bar{P})\cdot\phi dx=0.
\ea\ee
Integration by parts leads to
\be\la{e4} \ba
\int(P-\bar{P})\div\phi dx &= \left(\int\rho u\cdot\phi dx\right)_t-\int\rho u\cdot\nabla\phi\cdot udx - \int\rho u\cdot\phi_t dx \\
& \quad  +\mu\int\nabla u\cdot\nabla\phi dx + \int(\lambda(\rho)+\mu)(P-\bar{P})\div udx \\
& \leq \left(\int\rho u\cdot\phi dx\right)_t+C\|\rho^{\frac{1}{2}}u\|_{L^{4}}^{2}\|\nabla\phi\|_{L^2}+C\|\rho^{\frac{1}{2}}u\|_{L^2}\|\phi_t\|_{L^2}\\
& \quad +\mu\|\nabla u\|_{L^2}\|\nabla \phi\|_{L^2}+C(\hat{\rho})\|P-\bar{P}\|_{L^2}\|\nabla u\|_{L^2}.
\ea\ee
Therefore, utilizing Young's inequality and \eqref{e2}, we get
\be\la{e5} \ba
\int(P-\bar{P})^{2} dx\leq\left(\int\rho u\cdot\phi dx\right)_t + C\|\nabla u\|_{L^{2}}^{2}.
\ea\ee
Integrating \eqref{e5} over $(0,T]$ leads to
\be\la{e6} \ba
&\int_0^{T}\int(P-\bar{P})^{2} dxdt\\
& \leq C\left(\sup_{   0\le t\le T  }\|\rho^{\frac{1}{2}}u\|_{L^{2}}\|P-\bar{P}\|_{L^{2}}+\int_0^{T}\|\nabla u\|_{L^{2}}^{2}dt\right) \\
& \leq C(\hat{\rho})C_0^{1/2}.
\ea\ee

It is easy to check that
\be\la{xz2} \ba
(P-\bar{P})_t+u\cdot\nabla(P-\bar{P})+\gamma P\div u-(\gamma-1)\overline{(P-\bar{P})\div u}=0.
\ea\ee
due to $\eqref{a1}_1$.
Multiplying the above equation by $2\sigma(P-\bar{P})$ and integrating with respect to $x$ yields
\be\la{xz3} \ba
\left(\sigma\int(P-\bar{P})^2dx\right)_t\leq C(\sigma+\sigma')\int(P-\bar{P})^2dx+C\sigma\int|\nabla u|^2dx,
\ea\ee
which together with \eqref{bcgs1} and \eqref{e6} leads to
\be\la{xz4} \ba
\sup_{   0\le t\le T  }\sigma\|P-\bar{P}\|_{L^2}^2\leq CC_0^{1/2}.
\ea\ee
\eqref{a16}, \eqref{bcgs1}, \eqref{e4} and \eqref{xz4} imply
\be\la{xz5} \ba
&\int_0^T\sigma\|P-\bar{P}\|_{L^2}^2dt\\
&\leq\int_0^T\left(\sigma\int\rho u\cdot\phi dx\right)_tdt-\int_0^T\sigma'\int\rho u\cdot\phi dxdt+C\sigma\int_0^T\|\na u\|_{L^2}^2dt\\
&\leq C\sup_{   0\le t\le T  }\left(\int\rho|u|^2dx\right)^{1/2}\sup_{   0\le t\le T  }\left(\sigma\|P-\bar{P}\|_{L^2}^2\right)^{1/2}+C(\hat{\rho})C_0^{1/2}+CC_0\\
&\leq CC_0^{3/4}.
\ea\ee
\end{proof}

The following conclusion shows preliminary $L^{2}$ bounds for $\nabla u$ and $\rho^{1/2}\dot{u}$.
\begin{lemma}\la{xcrle1}
 Let $(\n,u)$ be a smooth solution of
 \eqref{a1}-\eqref{ch1} satisfying \eqref{zz1}.
  Then there is a positive constant
  $C $ depending only  on $\mu,$ $\lambda,$ $a$, $\gamma$, $\hat{\rho}$ and $\Omega$  such that
  \be\la{h14}
  A_1(T) \le  C C_0^{1/2} + C\int_0^{T}\int\sigma|\nabla u|^3dx dt,
  \ee
 and
  \be\la{h15}
    A_2(T)
    \le   C C_0^{2/3} + CA_1(T)  + C\int_0^{T}\int \sigma^3 |\nabla u|^4 dt.
   \ee
\end{lemma}
\begin{proof}
 Some ideas of the proof come from Hoff \cite{H3}.
 Let $m\ge 0$ be a real number which will be determined later.

 Multiplying $(\ref{a1})_2 $ by
$\sigma^m \dot{u}$   and then integrating the resulting equality over
$\Omega$ lead  to
\be\la{I0} \ba  \int \sigma^m \rho|\dot{u}|^2dx &
= -\int\sigma^m \dot{u}\cdot\nabla Pdx + \int\sigma^m \nabla((\lambda(\rho)+2\mu)\div u)\cdot\dot{u}dx \\
&\quad - \mu\int\sigma^m \nabla\times\curl u\cdot\dot{u}dx \\
& \triangleq I_1+I_2+I_3. \ea \ee

$I_1$, $I_2$ and $I_3$ will be estimated one by one.
By $(\ref{Pu1})$, we have
\be\la{I10} \ba
I_1 = & - \int \sigma^m \dot{u}\cdot\nabla Pdx \\
= & -\int\sigma^m u_{t}\cdot\nabla Pdx
- \int\sigma^m u\cdot\nabla u\cdot\nabla Pdx \\
= & \left(\int\sigma^m(P-\bar{P})\,\div u\, dx\right)_{t} - m\sigma^{m-1}\sigma'\int (P-\bar{P})\,\div u\, dx - \int \sigma^{m}\div u\, P_{t}dx \\
&- \int\sigma^{m}u\cdot\nabla u\cdot\nabla Pdx\\
= & \left(\int\sigma^m(P-\bar{P})\,\div u\, dx\right)_{t} - m\sigma^{m-1}\sigma'\int (P-\bar{P})\,\div u\,dx + \int\sigma^{m}P\nabla u:\nabla u dx \\
&+ (\gamma-1)\int\sigma^{m}P(\div u)^{2}dx - \int_{\partial\Omega}\sigma^{m}Pu\cdot\nabla u\cdot n ds. \ea \ee
For the last term in the last equality of \eqref{I10}, by \eqref{bdd1}, Sobolev trace theorem and Poincar\'{e}'s inequality, we get
\be\la{bdt1} \ba
&-\int_{\partial\Omega}\sigma^{m}Pu\cdot\nabla u\cdot n ds\\
&=\int_{\partial\Omega}\sigma^{m}Pu\cdot\nabla n\cdot uds \\
&\leq C\int_{\partial\Omega}\sigma^{m}|u|^{2}ds \\
& \leq C\sigma^{m}\|\nabla u\|_{L^{2}}^{2}.
\ea  \ee
Hence,
\be\la{I1} \ba
I_1 \leq &\left(\int\sigma^m(P-\bar{P})\,\div u\, dx\right)_{t} + C\sigma^{m}\|\nabla u\|_{L^{2}}^{2} + m\sigma^{m-1}\sigma'\|P-\bar{P}\|_{L^{2}}\|\nabla u\|_{L^{2}}\\
\leq &\left(\int\sigma^m(P-\bar{P})\,\div u\, dx\right)_{t} + C\|\nabla u\|_{L^{2}}^{2} + C(\hat{\rho})m\sigma^{m-1}\sigma'\|\nabla u\|_{L^{2}}.
\ea \ee

For $I_2,$ we have
\bnn \ba
I_2 & =  \int\sigma^m \nabla((\lambda(\rho)+2\mu)\div u)\cdot\dot{u}dx \\
& = \int_{\partial\Omega}\sigma^m(\lambda(\rho)+2\mu)\div u\,(\dot{u}\cdot n)ds - \int\sigma^m(\lambda(\rho)+2\mu)\div u\,\div \dot{u}dx  \\
& = \int_{\partial\Omega}\sigma^m(\lambda(\rho)+2\mu)\div u\,(u\cdot\nabla u\cdot n)ds - \frac{1}{2}\left(\int\sigma^{m}(\lambda(\rho)+2\mu)(\div u)^{2}dx\right)_{t} \\
&\quad - \int\sigma^m(\lambda(\rho)+2\mu)\div u\,\div(u\cdot\nabla u)dx + \frac{1}{2}m\sigma^{m-1}\sigma'\int(\lambda(\rho)+2\mu)(\div u)^{2}dx \\
&\quad + \frac{1}{2}\int\sigma^m(\lambda(\rho)+2\mu)_t(\div u)^{2}dx \\
& = - \frac{1}{2}\left(\int\sigma^{m}(\lambda(\rho)+2\mu)(\div u)^{2}dx\right)_{t}+\int_{\partial\Omega}\sigma^m(\lambda(\rho)+2\mu)\div u\,(u\cdot\nabla u\cdot n)ds \\
&\quad - \int\sigma^m(\lambda(\rho)+2\mu)\div u\,\partial_i(u^j\partial_ju^i)dx + \frac{1}{2}m\sigma^{m-1}\sigma'\int(\lambda(\rho)+2\mu)(\div u)^{2}dx \\
&\quad - \frac{1}{2}\int\sigma^m[\nabla\lambda(\rho)\cdot u+\lambda'(\rho)\rho\div u](\div u)^{2}dx\\
& = - \frac{1}{2}\left(\int\sigma^{m}(\lambda(\rho)+2\mu)(\div u)^{2}dx\right)_{t}+\int_{\partial\Omega}\sigma^m(\lambda(\rho)+2\mu)\div u\,(u\cdot\nabla u\cdot n)ds \\
&\quad - \int\sigma^m(\lambda(\rho)+2\mu)\div u\,\partial_iu^j\partial_ju^idx - \int\sigma^m(\lambda(\rho)+2\mu)\div u\,u^j\partial_{ji}u^idx \\
&\quad + \frac{1}{2}m\sigma^{m-1}\sigma'\int(\lambda(\rho)+2\mu)(\div u)^{2}dx \\
&\quad - \frac{1}{2}\int\sigma^m[\nabla\lambda(\rho)\cdot u+\lambda'(\rho)\rho\div u](\div u)^{2}dx\\
& = - \frac{1}{2}\left(\int\sigma^{m}(\lambda(\rho)+2\mu)(\div u)^{2}dx\right)_{t}+\int_{\partial\Omega}\sigma^m(\lambda(\rho)+2\mu)\div u\,(u\cdot\nabla u\cdot n)ds \\
&\quad - \int\sigma^m(\lambda(\rho)+2\mu)\div u\,\partial_iu^j\partial_ju^idx + \frac{1}{2}m\sigma^{m-1}\sigma'\int(\lambda(\rho)+2\mu)(\div u)^{2}dx \\
&\quad  + \frac{1}{2}\int\sigma^m(\lambda(\rho)+2\mu)(\div u)^{3}dx - \frac{1}{2}\beta\int\sigma^m\lambda(\rho)(\div u)^{3}dx\\
& \leq - \frac{1}{2}\left(\int\sigma^{m}(\lambda(\rho)+2\mu)(\div u)^{2}dx\right)_{t}+\int_{\partial\Omega}\sigma^m(\lambda(\rho)+2\mu)\div u\,(u\cdot\nabla u\cdot n)ds \\
&\quad  + C(\hat{\rho})\int\sigma^m|\nabla u|^{3}dx + C(\hat{\rho})\int|\nabla u|^{2}dx.
\ea  \enn

Just as what we have done in \eqref{bdt1}, Sobolev trace theorem, Poincar\'{e}'s inequality and Lemma \ref{le3} assert that
\bnn \ba
&\left|\int_{\partial\Omega}(\lambda(\rho)+2\mu)\div u\,(u\cdot\nabla u\cdot n)ds\right| \\
&=\left|\int_{\partial\Omega}(\lambda(\rho)+2\mu)\div u\,(u\cdot\nabla n\cdot u)ds\right|\\
& \leq\left|\int_{\partial\Omega}Fu\cdot\nabla n\cdot uds\right|+\left|\int_{\partial\Omega}(P-\bar{P})u\cdot\nabla n\cdot uds\right|\\
& \leq C\left(\int_{\partial\Omega}|F||u|^{2}ds+\int_{\partial\Omega}|u|^{2}ds\right)\\
& \leq C(\|\nabla F\|_{L^{2}}\|u\|_{L^{4}}^{2}+\|F\|_{L^{6}}\|u\|_{L^{3}}\|\nabla u\|_{L^{2}}+\|F\|_{L^{2}}\| u\|^2_{L^{4}}+\|\nabla u\|_{L^{2}}\|u\|_{L^{2}})\\
& \leq\frac{1}{2}\|\rho^{\frac{1}{2}}\dot{u}\|_{L^{2}}^{2}+C(\|\nabla u\|_{L^{2}}^{2}+\|\nabla u\|_{L^{2}}^{4}).
\ea  \enn
Therefore,
\be\la{I2} \ba
I_2 & \leq - \frac{1}{2}\left(\int\sigma^{m}(\lambda(\rho)+2\mu)(\div u)^{2}dx\right)_{t}+C\sigma^{m}\|\nabla u\|_{L^{3}}^{3}\\
&\quad +\frac{1}{2}\sigma^{m}\|\rho^{\frac{1}{2}}\dot{u}\|_{L^{2}}^{2}+C\sigma^{m}\|\nabla u\|_{L^{2}}^{4}+C\|\nabla u\|_{L^{2}}^{2}.\ea\ee

Finally, a direct computation shows that
\be\la{I3}\ba
I_3 & = -\mu\int\sigma^{m}\nabla\times\curl u\cdot\dot{u}dx \\
& =  - \mu\int\sigma^{m}\curl u\cdot\curl\dot{u}dx \\
& = -\frac{\mu}{2}\left(\int\sigma^{m}|\curl u|^{2}dx\right)_t + \frac{\mu m}{2}\sigma^{m-1}\sigma'\int|\curl u|^{2}dx \\
& \quad - \mu\int\sigma^{m}\curl u\cdot\curl(u\cdot\nabla u)dx \\
& = -\frac{\mu}{2}\left(\int\sigma^{m}|\curl u|^{2}dx\right)_t + \frac{\mu m}{2}\sigma^{m-1}\sigma'\int|\curl u|^{2}dx   \\
& \quad - \mu\int\sigma^{m}(\nabla u^{i}\times\nabla_i u)\cdot\curl udx + \frac{\mu}{2}\int\sigma^{m}|\curl u|^{2}\,\div udx\\
& \leq -\frac{\mu}{2}\left(\int\sigma^{m}|\curl u|^{2}dx\right)_t + Cm\sigma^{m-1}\sigma'\|\nabla u\|_{L^{2}}^{2} + C\sigma^{m}\|\nabla u\|_{L^{3}}^{3}.
\ea \ee
By \eqref{I1}-\eqref{I3}, $(\ref{I0})$, we obtain
\be\la{I4}\ba
&\left(\frac{1}{2}\int\sigma^{m}(\lambda(\rho)+2\mu)(\div u)^{2}dx+\mu\int\sigma^{m}|\curl u|^{2}dx\right)_{t}+\int\sigma^{m}\rho|\dot{u}|^{2}dx \\
& \leq \left(\int\sigma^{m}(P-\bar{P})\,\div udx\right)_{t}+Cm\sigma^{m-1}\sigma'\|\nabla u\|_{L^{2}}+C\sigma^{m}\|\nabla u\|_{L^{2}}^{4}\\
& \quad +C\|\nabla u\|_{L^{2}}^{2}+C\sigma^{m}\|\nabla u\|_{L^{3}}^{3}.
\ea \ee
Integrating over $(0,T]$, by \eqref{tdu1}, Young's inequality, Lemmas \ref{le2} and \ref{th00}, we conclude that for any $m\geq1$,
\be\la{I40}\ba
&\sigma^{m}\|\nabla u\|_{L^{2}}^{2}+\int_0^T\int\sigma^{m}\rho|\dot{u}|^{2}dxdt \\
& \leq CC_{0}^{1/2}+C\int_0^T\sigma^{m}\|\nabla u\|_{L^{2}}^{4}dt+C\int_0^T\sigma^{m}\|\nabla u\|_{L^{3}}^{3}dt.
\ea \ee
 Choose $m=1,$ together with the assumption \eqref{zz1} and \eqref{bcgs1}, we obtain \eqref{h14}.

Next, \eqref{h15} will be proved. Rewrite $ (\ref{a1})_2 $ in the form
\be\la{xdy1}\ba
\rho\dot{u}=\nabla F - \mu\nabla\times\curl u.
\ea \ee
Operating $ \sigma^{m}\dot{u}^{j}[\pa/\pa t+\div
(u\cdot)] $ to $ (\ref{xdy1})^j,$ summing with respect to $j$, and integrating over $\Omega,$ together with $ (\ref{a1})_1 $, we get
\be\la{ax1}\ba &\left(\frac{\sigma^{m}}{2}\int\rho|\dot{u}|^{2}dx\right)_t-\frac{m}{2}\sigma^{m-1}\sigma'\int\rho|\dot{u}|^{2}dx \\
& = -\int\sigma^{m}\dot{u}^{j}\,\div(\rho\dot{u}^{j}u)dx - \frac{1}{2}\int\sigma^{m}\rho_t|\dot{u}|^{2}dx
+\int\sigma^{m}(\dot{u}\cdot\nabla F_t+\dot{u}^{j}\,\div(u\partial_jF))dx  \\
&\quad-\mu\int\sigma^{m}(\dot{u}\cdot\nabla\times\curl u_t+\dot{u}^{j}\,\div((\nabla\times\curl u)^j\,u))dx \\
& = \int\sigma^{m}(\dot{u}\cdot\nabla F_t+\dot{u}^{j}\,\div(u\partial_jF))dx\\
&\quad+\mu\int\sigma^{m}(-\dot{u}\cdot\nabla\times\curl u_t-\dot{u}^{j}\div((\nabla\times\curl u)^j\,u))dx \\
& \triangleq J_1+J_2.
\ea\ee

Let us estimate $J_1, J_2$.

By \eqref{Pu2}, a direct computation yields
\be\la{ax2}\ba J_1 & =\int\sigma^{m}\dot{u}\cdot\nabla F_tdx+\int\sigma^{m}\dot{u}^{j}\div(u\partial_jF)dx \\
& = \int_{\partial\Omega}\sigma^{m}F_t\dot{u}\cdot nds-\int\sigma^{m}F_t\,\div\dot{u}dx-\int\sigma^{m}u\cdot\nabla\dot{u}^j\partial_jFdx\\
& = \int_{\partial\Omega}\sigma^{m}F_t\dot{u}\cdot nds +\beta\int\sigma^{m}\lambda(\rho)(\div u)^2\,\div\dot{u}dx-\int\sigma^{m}(\lambda(\rho)+2\mu)(\div\dot{u})^2dx\\
&\quad+\int\sigma^{m}(\lambda(\rho)+2\mu)\nabla u:\nabla u\,\div\dot{u}dx-\gamma\int\sigma^{m} P\div u\,\div\dot{u}dx\\
&\quad+\int\sigma^{m}\nabla F\cdot u\,\div\dot{u}dx-\int\sigma^{m}u\cdot\nabla\dot{u}^j\partial_jFdx \\
&\leq\int_{\partial\Omega}\sigma^{m}F_t\dot{u}\cdot nds-\int\sigma^{m}(\lambda(\rho)+2\mu)(\div\dot{u})^2dx+\frac{\delta}{4}\sigma^m\|\nabla\dot{u}\|_{L^2}^2+\delta\sigma^m\|\nabla F\|_{L^6}^2\\
&\quad+C\sigma^{m}(\|\nabla u\|_{L^2}^4\|\nabla F\|_{L^2}^2+\|\nabla u\|_{L^4}^4+\|\nabla u\|_{L^2}^2)
\ea\ee
where in the third equality we have used
\bnn\ba
F_t&=[(\lambda(\rho)+2\mu)\div u-(P-\bar{P})]_t\\
&=(\lambda(\rho))_t\div u+(\lambda(\rho)+2\mu)\div u_t-P_t\\
&=\lambda'(\rho)\rho_t\div u+(\lambda(\rho)+2\mu)\div\dot{u}-(\lambda(\rho)+2\mu)\div(u\cdot\nabla u)+u\cdot\nabla P+\gamma P\div u\\
&=-\lambda'(\rho)\div(\rho u)\div u+(\lambda(\rho)+2\mu)\div\dot{u}-(\lambda(\rho)+2\mu)\div(u\cdot\nabla u)\\
&\quad+u\cdot\nabla P+\gamma P\div u\\
&=-\beta\lambda(\rho)(\div u)^2-\nabla\lambda(\rho)\cdot u\div u+(\lambda(\rho)+2\mu)\div\dot{u}-(\lambda(\rho)+2\mu)u\cdot\nabla\div u\\
&\quad-(\lambda(\rho)+2\mu)\nabla u:\nabla u+u\cdot\nabla P+\gamma P\div u\\
&=-\beta\lambda(\rho)(\div u)^2+(\lambda(\rho)+2\mu)\div\dot{u}-(\lambda(\rho)+2\mu)\nabla u:\nabla u+\gamma P\div u\\
&\quad-[\div u\nabla\lambda(\rho)+(\lambda(\rho)+2\mu)\nabla\div u-\nabla P]\cdot u\\
&=-\beta\lambda(\rho)(\div u)^2+(\lambda(\rho)+2\mu)\div\dot{u}-(\lambda(\rho)+2\mu)\nabla u:\nabla u+\gamma P\div u\\
&\quad-\nabla F\cdot u\\
\ea\enn

Denote $h\triangleq u\cdot(\nabla n+(\nabla n)^{tr}),\,\,u^\bot=-u\times n.$

For the boundary term in \eqref{ax2}, we have
\be\la{ax3}\ba
&\int_{\partial\Omega}\sigma^{m}F_t\dot{u}\cdot nds=-\int_{\partial\Omega}\sigma^{m}F_t\,(u\cdot\nabla n\cdot u)ds \\
& = -\left(\int_{\partial\Omega}\sigma^{m}(u\cdot\nabla n\cdot u)Fds\right)_t+\int_{\partial\Omega}\sigma^{m}Fh\cdot\dot{u}ds -\int_{\partial\Omega}\sigma^{m}Fh\cdot(u\cdot\nabla u)ds\\
&\quad+m\sigma^{m-1}\sigma'\int_{\partial\Omega}(u\cdot\nabla n\cdot u)Fds \\
& = -\left(\int_{\partial\Omega}\sigma^{m}(u\cdot\nabla n\cdot u)Fds\right)_t+\int_{\partial\Omega}\sigma^{m}Fh\cdot\dot{u}ds+m\sigma^{m-1}\sigma'\int_{\partial\Omega}(u\cdot\nabla n\cdot u)Fds \\
&\quad- \sum_{i=1}^{3}\int_{\partial\Omega}\sigma^{m}Fh^{i}(\nabla u^{i}\times u^{\perp})\cdot nds \\
& = -\left(\int_{\partial\Omega}\sigma^{m}(u\cdot\nabla n\cdot u)Fds\right)_t+\int_{\partial\Omega}\sigma^{m}Fh\cdot\dot{u}ds+m\sigma^{m-1}\sigma'\int_{\partial\Omega}(u\cdot\nabla n\cdot u)Fds\\
&\quad-\sum_{i=1}^{3}\int\sigma^{m}\nabla u^i\times u^\perp\cdot\nabla(Fh^i)dx+\sum_{i=1}^{3}\int\sigma^{m}Fh^i\nabla\times u^\perp \cdot\nabla u^{i}dx \\
& \leq -\left(\int_{\partial\Omega}\sigma^{m}(u\cdot\nabla n\cdot u)Fds\right)_t+C\sigma^{m}(\|\nabla F\|_{L^2}\|u\|_{L^6}\|\dot{u}\|_{L^3})\\
&\quad+C\sigma^{m}(\| F\|_{L^3}\|\nabla u\|_{L^2}\|\dot{u}\|_{L^6}+\|F\|_{L^3}\|u\|_{L^6}\|\nabla\dot{u}\|_{L^2}+\|F\|_{L^3}\|u\|_{L^6}\|\dot{u}\|_{L^2})\\
&\quad+C\sigma^{m}(\|\nabla u\|_{L^2}\|u\|_{L^6}^{2}\|\nabla F\|_{L^6}+\|\nabla u\|_{L^4}^2\|u\|_{L^6}\|F\|_{L^3})\\
&\quad+m\sigma^{m-1}\sigma'(\|\nabla u\|_{L^2}\|u\|_{L^6}\|F\|_{L^3}+\|u\|_{L^4}^2\|F\|_{L^2}+\|\nabla F\|_{L^2}\|u\|_{L^4}^2)\\
&\quad+C\sigma^{m}\left((\|\rho\dot{u}\|_{L^2}+\|P-\bar{P}\|_{L^2}+\|\nabla u\|_{L^2}+\|\nabla u\|_{L^2}^2)\|\nabla u\|_{L^2}(\|\nabla\dot{u}\|_{L^2}+\|\nabla u\|_{L^4}^2\right)\\
&\quad+m\sigma^{m-1}\sigma'\|\nabla u\|_{L^2}^2(\|\rho\dot{u}\|_{L^2}+\|P-\bar{P}\|_{L^2}+\|\nabla u\|_{L^2})\\
&\le-\left(\int_{\partial\Omega}\sigma^{m}(u\cdot\nabla n\cdot u)Fds\right)_t+\frac{\delta}{4}\sigma^{m}\|\nabla\dot{u}\|_{L^2}^2+C\sigma^{m}(\|\rho^{\frac{1}{2}}\dot{u}\|_{L^2}^{2}\|\nabla u\|_{L^2}^{2}+\|\nabla u\|_{L^4}^{4})\\
&\quad+C\sigma^{m}(\|\nabla u\|_{L^2}^{2}+\|\nabla u\|_{L^2}^{6})+Cm\sigma^{m-1}\sigma'(\|\rho^{\frac{1}{2}}\dot{u}\|_{L^2}^{2}+\|\nabla u\|_{L^2}^{2}+\|\nabla u\|_{L^2}^{4}).
\ea\ee
Together with \eqref{ax2} and \eqref{ax3}, one has
\be\la{ax399}\ba
& J_1 \leq Cm\sigma^{m-1}\sigma'(\|\rho^{\frac{1}{2}}\dot{u}\|_{L^2}^2+\|\nabla u\|_{L^2}^2+\|\nabla u\|_{L^2}^4)-\left(\int_{\partial\Omega}\sigma^{m}(u\cdot\nabla n\cdot u)Fds\right)_t \\
&\quad+\frac{\delta}{2}\sigma^{m}\|\nabla\dot{u}\|_{L^2}^2+C\sigma^{m}\|\rho^{\frac{1}{2}}\dot{u}\|_{L^2}^2\|\nabla u\|_{L^2}^2-\int\sigma^{m}(\lambda(\rho)+2\mu)(\div\dot{u})^2dx\\
& \quad+\delta\sigma^m\|\nabla F\|_{L^6}^2+C\sigma^{m}(\|\nabla u\|_{L^2}^4\|\nabla F\|_{L^2}^2+\|\nabla u\|_{L^2}^2+\|\nabla u\|_{L^2}^6+\|\nabla u\|_{L^4}^4).
\ea\ee
Similarly, a straightforward calculation leads to
\be\la{ax3999}\ba
J_2&=-\mu\int\sigma^{m}\dot{u}\cdot(\nabla\times\curl u_t)dx-\mu\int\sigma^{m}\dot{u}\cdot(\nabla\times\curl u)\,\div udx\\
&\quad-\mu\int\sigma^{m} u^{i}\dot{u}\cdot\nabla\times(\nabla_i\curl u)dx  \\
&=-\mu\int\sigma^{m}|\curl \dot{u}|^{2}dx+\mu\int\sigma^{m}\curl \dot{u}\cdot\curl(u\cdot\nabla u)dx \\
&\quad+\mu\int\sigma^{m}(\curl u\times\dot{u})\cdot\nabla\div udx -\mu\int\sigma^{m}\div u\,\curl u\cdot\curl \dot{u}dx\\
&\quad-\mu\int\sigma^{m} u^{i}\div(\nabla_i\curl u\times\dot{u})dx-\mu\int\sigma^{m} u^i\nabla_i\curl u\cdot\curl\dot{u}dx  \\
&=-\mu\int\sigma^{m}|\curl\dot{u}|^{2}dx+\mu\int\sigma^{m}\curl\dot{u}\cdot(\nabla u^i\times\nabla_i u) dx \\
&\quad+\mu\int\sigma^{m}(\curl u\times\dot{u})\cdot\nabla\div udx-\mu\int\sigma^{m}\div u\,\curl u\cdot\curl \dot{u}dx \\
&\quad-\mu\int\sigma^{m}  u\cdot\nabla\div(\curl u\times\dot{u})dx+\mu\int\sigma^{m} u^i\div(\curl u\times\nabla_i\dot{u})dx \\
&=-\mu\int\sigma^{m}|\curl \dot{u}|^{2}dx+\mu\int\sigma^{m}(\curl u\times\nabla u^i)\cdot\nabla_i\dot{u}dx \\
&\quad+\mu\int\sigma^{m}\curl\dot{u}\cdot(\nabla u^i\times\nabla_i u)dx-\mu\int\sigma^{m}\div u\,\curl u\cdot\curl \dot{u}dx\\
&\leq \frac{\delta}{2}\sigma^{m}\|\nabla\dot{u}\|_{L^2}^2+C\sigma^{m}\|\nabla u\|_{L^4}^4-\mu\int\sigma^{m}|\curl \dot{u}|^{2}dx.
\ea\ee

Combining \eqref{ax399} with \eqref{ax3999}, we deduce from \eqref{ax1} that

\be\la{ax40}\ba
&\left(\frac{\sigma^{m}}{2}\|\rho^{\frac{1}{2}}\dot{u}\|_{L^2}^2\right)_t+\int\sigma^{m}(\lambda(\rho)+2\mu)(\div\dot{u})^2dx+\mu\sigma^{m}\|\curl\dot{u}\|_{L^2}^2\\
& \leq Cm\sigma^{m-1}\sigma'(\|\rho^{\frac{1}{2}}\dot{u}\|_{L^2}^2+\|\nabla u\|_{L^2}^2+\|\nabla u\|_{L^2}^4)-\left(\int_{\partial\Omega}\sigma^{m}(u\cdot\nabla n\cdot u)Fds\right)_t \\
&\quad+\delta\sigma^{m}\|\nabla\dot{u}\|_{L^2}^2+C\sigma^{m}\|\rho^{\frac{1}{2}}\dot{u}\|_{L^2}^2\|\nabla u\|_{L^2}^2+\delta\sigma^{m}\|\nabla F\|_{L^6}^2\\
& \quad+C\sigma^{m}(\|\nabla u\|_{L^2}^4\|\nabla F\|_{L^2}^2+\|\nabla u\|_{L^2}^2+\|\nabla u\|_{L^2}^6+\|\nabla u\|_{L^4}^4).
\ea\ee
By \eqref{tb11}, \eqref{tb90}, \eqref{h19} and choose $\delta$ small enough, it follows from \eqref{ax40} that
\be\la{ax401}\ba
&\left(\sigma^{m}\|\rho^{\frac{1}{2}}\dot{u}\|_{L^2}^2\right)_t+\sigma^{m}\|\nabla\dot{u}\|_{L^2}^2\\
& \leq Cm\sigma^{m-1}\sigma'(\|\rho^{\frac{1}{2}}\dot{u}\|_{L^2}^2+\|\nabla u\|_{L^2}^2+\|\nabla u\|_{L^2}^4)-\left(\int_{\partial\Omega}\sigma^{m}(u\cdot\nabla n\cdot u)Fds\right)_t \\
&\quad+C\sigma^{m}(\|\nabla u\|_{L^2}^4\|\rho^{\frac{1}{2}}\dot{u}\|_{L^2}^2+\|\nabla u\|_{L^2}^2+\|\nabla u\|_{L^2}^6+\|\nabla u\|_{L^4}^4)\\
&\quad+C\sigma^{m}\|\rho^{\frac{1}{2}}\dot{u}\|_{L^2}^2\|\nabla u\|_{L^2}^2.
\ea\ee
Integrating over $(0,T]$, for $m\geq1$, we get
\be\la{ax402}\ba
&\sigma^{m}\|\rho^{\frac{1}{2}}\dot{u}\|_{L^2}^2+\int_0^T\sigma^{m}\|\nabla\dot{u}\|_{L^2}^2dt\\
& \leq C\int_0^{\sigma(T)}m\sigma^{m-1}(\|\rho^{\frac{1}{2}}\dot{u}\|_{L^2}^2+\|\nabla u\|_{L^2}^2+\|\nabla u\|_{L^2}^4)dt-\int_{\partial\Omega}\sigma^{m}(u\cdot\nabla n\cdot u)Fds \\
&\quad+C\int_0^T\sigma^{m}\|\rho^{\frac{1}{2}}\dot{u}\|_{L^2}^2\|\nabla u\|_{L^2}^2dt+C\int_0^T\sigma^{m}(\|\nabla u\|_{L^2}^2+\|\nabla u\|_{L^2}^6+\|\nabla u\|_{L^4}^4)dt\\
&\quad+C\int_0^T\sigma^{m}\|\nabla u\|_{L^2}^4\|\rho^{\frac{1}{2}}\dot{u}\|_{L^2}^2dt\\
& \leq C\int_0^{\sigma(T)}m\sigma^{m-1}(\|\rho^{\frac{1}{2}}\dot{u}\|_{L^2}^2+\|\nabla u\|_{L^2}^2+\|\nabla u\|_{L^2}^4)dt+\frac{\sigma^{m}}{2}\|\rho^{\frac{1}{2}}\dot{u}\|_{L^2}^2\\
&\quad +C\int_0^T\sigma^{m}\|\rho^{\frac{1}{2}}\dot{u}\|_{L^2}^2\|\nabla u\|_{L^2}^2dt+C\int_0^T\sigma^{m}(\|\nabla u\|_{L^2}^2+\|\nabla u\|_{L^2}^6)dt\\
& \quad +C\int_0^T\sigma^{m}\|\nabla u\|_{L^4}^4dt+C\sigma^{m}(\|\nabla u\|_{L^2}^2+\|\nabla u\|_{L^2}^4)+C\int_0^T\sigma^{m}\|\nabla u\|_{L^2}^4\|\rho^{\frac{1}{2}}\dot{u}\|_{L^2}^2dt,
\ea\ee
where we have used the fact:
\be\la{Fnn1}\ba
&\left|\int_{\partial\Omega}(u\cdot\nabla n\cdot u)Fds\right|\\
&\leq C \int_{\partial\Omega}|u|^2|F|ds\\
& \leq C(\|\nabla F\|_{L^2}\|\nabla u\|_{L^2}^2
+\|F\|_{L^6}\|u\|_{L^3}\|\nabla u\|_{L^2}+\|F\|_{L^2}\|u\|_{L^4}^2)\\
&\leq\frac{1}{2}\|\rho^{\frac{1}{2}}\dot{u}\|_{L^2}^2+C(\|\nabla u\|_{L^2}^2+\|\nabla u\|_{L^2}^4).
\ea\ee
Therefore,
\be\la{ax7}\ba
&\sigma^{m}\|\rho^{\frac{1}{2}}\dot{u}\|_{L^2}^2+\int_0^T\sigma^{m}\|\nabla\dot{u}\|_{L^2}^2dt\\
& \leq C\int_0^{\sigma(T)}m\sigma^{m-1}(\|\rho^{\frac{1}{2}}\dot{u}\|_{L^2}^2+\|\nabla u\|_{L^2}^2+\|\nabla u\|_{L^2}^4)dt+C\int_0^T\sigma^{m}\|\nabla u\|_{L^4}^4dt\\
&\quad +C\int_0^T\sigma^{m}\|\rho^{\frac{1}{2}}\dot{u}\|_{L^2}^2\|\nabla u\|_{L^2}^2dt+C\int_0^T\sigma^{m}(\|\nabla u\|_{L^2}^2+\|\nabla u\|_{L^2}^6)dt\\
& \quad +C\sigma^{m}(\|\nabla u\|_{L^2}^2+\|\nabla u\|_{L^2}^4)+C\int_0^T\sigma^{m}\|\nabla u\|_{L^2}^4\|\rho^{\frac{1}{2}}\dot{u}\|_{L^2}^2dt,
\ea\ee

Now take $m=3$ in \eqref{ax7}, by \eqref{zz1} and Lemma \ref{le2}, we give \eqref{h15} and complete the proof of Lemma \ref{xcrle1}.
\end{proof}

\begin{lemma}\la{nzc1} If $(\n,u)$ is a smooth solution of \eqref{a1}-\eqref{ch1} satisfying \eqref{zz1}, then there
exists a positive constant  $\varepsilon_0$   depending only on $\mu,\,  \lambda,\,   \ga,\,  a,\,  \hat{\rho},\, \Omega$ and $M$ such
that
\be\la{xuv1} A_3(\sigma(T))+\int_0^{\sigma(T)}\|\rho^{1/2}\dot{u}\|^2dt\le 2K,\ee
provided $A_3(\sigma(T))\leq3K$ and $C_0\leq\varepsilon_0$.
\end{lemma}
\begin{proof}
Choosing $m=0$ in \eqref{I40} and using \eqref{bcgs1} lead  to
\be\la{xinl2}\ba &A_3(\sigma(T))+\int_0^{\sigma(T)}\|\rho^{1/2}\dot{u}\|_{L^2}^2dt\\
& \leq CC_{0}+C\int_0^{\sigma(T)}\|\nabla u\|_{L^{2}}^{4}dt+C\int_0^{\sigma(T)}\|\nabla u\|_{L^{3}}^{3}dt\\
&\le C(C_0^{1/2}+M+KC_0)+C\int_0^{\sigma(T)}\|\nabla u\|_{L^3}^3dt,\ea\ee
By \eqref{h18} and \eqref{bcgs1}, one has
\bnn\ba
C\int_0^{\sigma(T)}\|\nabla u\|_{L^{3}}^{3}dt&\leq C\int_0^{\sigma(T)}\|\rho^{1/2}\dot{u}\|_{L^2}^{\frac{3}{2}}(\|\nabla u\|_{L^2}^{\frac{3}{2}}+\|P-\bar{P}\|_{L^2}^{\frac{3}{2}})dt\\
&\quad+C\int_0^{\sigma(T)}(\|\nabla u\|_{L^2}^{3}+\|P-\bar{P}\|_{L^2}^{3}+\|P-\bar{P}\|_{L^3}^{3})dt\\
&\le\frac{1}{2}\int_0^{\sigma(T)}\|\rho^{1/2}\dot{u}\|_{L^2}^{2}dt+C\int_0^{\sigma(T)}(\|\nabla u\|_{L^{2}}^{6}+\|P-\bar{P}\|_{L^{2}}^{6})dt\\
&\quad+C\int_0^{\sigma(T)}(\|\nabla u\|_{L^2}^{3}+\|P-\bar{P}\|_{L^2}^{3}+\|P-\bar{P}\|_{L^3}^{3})dt,
\ea\enn
which along with \eqref{xinl2}, yields that
\be\la{xinl3}\ba &A_3(\sigma(T))+\int_0^{\sigma(T)}\|\rho^{1/2}\dot{u}\|^2dt\\
&\le C(C_0+M+KC_0)+CC_0[A_3(\sigma(T))]^2+CC_0[A_3(\sigma(T))]^{\frac{1}{2}}\\
&\le K+CC_0[A_3(\sigma(T))]^2,
\ea\ee
for some suitably large positive $K$ depending only $\lambda,\,\,\mu,\,\,a,\,\,\hat{\rho}$ and $M$. Choosing $\varepsilon_0\triangleq(9CK)^{-1}$, one thus finishes the proof of \eqref{xuv1}.
\end{proof}
\begin{lemma}\la{le5} Let $(\n,u)$ be a smooth solution  of
   \eqref{a1}-\eqref{ch1}     on $\O \times (0,T] $ satisfying \eqref{zz1} and the initial data condition in \eqref{dt2}. Then there exists a positive constant $\varepsilon_1$ depending only  on $\mu,$  $\lambda,$  $\gamma,$ $a$, $\hat{\rho}$, $M$ and $\Omega$
 such that
  \be\la{h27}
  A_1(T)+A_2(T)\le C_0^{1/3},
  \ee
 provided $C_0\leq\varepsilon_1$.
   \end{lemma}

\begin{proof}
Lemma \ref{xcrle1} shows that
\be\la{h9910} \ba
A_1(T)+A_2(T)\leq CC_0^{1/2}+C\left(\int_0^{T}\sigma^3 \|\na u\|_{L^4}^4 dt+\int_0^T\sigma\|\nabla u\|_{L^{3}}^{3}dt\right)   .
\ea \ee
By (\ref{h18}), \eqref{zz1} and Lemmas \ref{le2}, \ref{th00}, one can check that
  \be\la{h99} \ba
  &  \int_0^{T}\sigma^3 \|\na u\|_{L^4}^4 dt\\
& \le  C \int_0^{T}\sigma^{3} \|\n^{1/2}  \dot u \|_{L^2}^3(\|\nabla u\|_{L^{2}}+\|P-\bar{P}\|_{L^{2}})dt\\
& \quad + C\int_{0}^{T}\sigma^{3}(\|\nabla u\|_{L^{2}}^{4}+\|P-\bar{P}\|_{L^{2}}^{4}+\|P-\bar{P}\|_{L^{4}}^{4})dt \\
& \le C\int_{0}^{T}(\sigma\|\rho^{1/2}\dot{u}\|_{L^{2}}^{2})(\sigma^{3}\|\rho^{1/2}\dot{u}\|_{L^{2}}^{2})^{1/2}(\sigma\|\nabla u\|_{L^{2}}^{2})^{1/2}dt\\
&\quad+C\int_{0}^{T}(\sigma\|\rho^{1/2}\dot{u}\|_{L^{2}}^{2})(\sigma^{3}\|\rho^{1/2}\dot{u}\|_{L^{2}}^{2})^{1/2}(\sigma^{1/2}\|P-\bar{P}\|_{L^{2}})dt \\
& \quad +C\left(\int_{0}^{T}(\sigma\|\nabla u\|_{L^{2}}^{2})\|\nabla u\|_{L^{2}}^{2}dt+\int_0^T\sigma^{3}\|P-\bar{P}\|_{L^{2}}^{2}dt\right) \\
& \le C\left[(A_1^{1/2}(T)+C_0^{\frac{1}{4}})A_2^{1/2}(T)A_1(T)+C_0^{3/4}\right] \\
& \le CC_0^{2/3},
 \ea \ee
which, along with \eqref{h9910} gives
\be\la{h991} \ba
A_1(T)+A_2(T)\leq CC_0^{1/2}+C\int_0^T\sigma\|\nabla u\|_{L^{3}}^{3}dt.
\ea \ee
So it reduces to estimate $\int_0^T\sigma\|\nabla u\|_{L^{3}}^{3}dt .$

By (\ref{h18}), \eqref{zz1} and Lemmas \ref{le2}, \ref{th00}, we get
\be\la{h992} \ba
&\int_0^{\sigma(T)}\sigma\|\nabla u\|_{L^{3}}^{3}dt \\
& \le C\int_0^{\sigma(T)}\sigma\|\rho^{1/2}\dot{u}\|_{L^{2}}^{3/2}(\|\nabla u\|_{L^{2}}^{3/2}+\|P-\bar{P}\|_{L^{2}}^{3/2})dt\\
& \quad +C\int_0^{\sigma(T)}\sigma(\|\nabla u\|_{L^{2}}^{3}+\|P-\bar{P}\|_{L^{2}}^{3}+\|P-\bar{P}\|_{L^{3}}^{3})dt\\
& \le C\int_0^{\sigma(T)}\sigma\|\rho^{1/2}\dot{u}\|_{L^{2}}^{3/2}\|\nabla u\|_{L^{2}}^{3/2}dt+C\int_0^{\sigma(T)}\sigma\|\rho^{1/2}\dot{u}\|_{L^{2}}^{3/2}\|P-\bar{P}\|_{L^{2}}^{3/2}dt+C(\hat{\rho})C_0^{3/4}\\
& \le C\sup_{   0\le t\le T  }(\sigma\|\nabla u\|_{L^{2}}^{2})^{1/4}\|\nabla u\|_{L^{2}}^{1/2}\int_0^{\sigma(T)}\|\nabla u\|_{L^{2}}^{1/2}(\sigma\|\rho^{1/2}\dot{u}\|_{L^{2}}^{2})^{3/4}dt\\
& \quad + C\left(\int_0^{\sigma(T)}\sigma\|\rho^{1/2}\dot{u}\|_{L^{2}}^{2}dt\right)^{3/4}\left(\int_0^{\sigma(T)}\sigma\|P-\bar{P}\|_{L^{2}}^{6}dt\right)^{1/4}+C(\hat{\rho})C_0^{3/4}\\
& \le C(\hat{\rho},M)C_0^{1/2},
\ea \ee
provided $C_0\leq\varepsilon_0$.

On the other hand, by \eqref{h99} and \eqref{bcgs1},
\be\la{h993} \ba
&\int_{\sigma(T)}^T\sigma\|\nabla u\|_{L^{3}}^{3}dt \\
& \le \int_{\sigma(T)}^T\sigma\|\nabla u\|_{L^{4}}^{4}dt+\int_{\sigma(T)}^T\sigma\|\nabla u\|_{L^{2}}^{2}dt\\
& \le CC_0^{2/3}.
\ea \ee

Set $\varepsilon_1\triangleq\min\{\varepsilon_0,(C(\hat{\rho}, M))^{-6}\}$, then for $C_0\leq\varepsilon_1,$ $C(\hat{\rho},M)C_0^{1/2}\leq C_0^{1/3}$. Together with \eqref{h992} and \eqref{h993}, it follows from \eqref{h991} that \eqref{h27} holds when $C_0\leq\varepsilon_1$.
\end{proof}
\begin{lemma}\la{le20} Let $(\n,u)$ be a smooth solution  of
   \eqref{a1}-\eqref{ch1}     on $\O \times (0,T] $ satisfying \eqref{zz1} and the initial data condition \eqref{dt2}. Then there exists a positive constant $C$ depending only  on $\mu,$  $\lambda,$  $\gamma,$ $a,$ $\hat{\rho}$, $M$, and $\Omega$
 such that
\be\la{xinl4}
\sup_{  0\le t\le T   }\|\nabla u\|_{L^2}^2+\int_0^{T}\|\rho^{1/2}\dot{u}\|_{L^{2}}^{2}dt\le C(\hat{\rho},M),
\ee
\be\la{xinl5}
\sup_{  0\le t\le T   }\sigma\|\rho^{1/2}\dot{u}\|_{L^{2}}^{2}+\int_0^{T}\sigma\|\nabla\dot{u}\|_{L^{2}}^{2}dt\le C(\hat{\rho},M),
\ee
 provided $C_0\leq\varepsilon_1$.
\end{lemma}
\begin{proof}
\eqref{xinl4} is a direct consequence of \eqref{xuv1} and \eqref{h27}.

Next, we prove \eqref{xinl5}. Choosing $m=1$ in \eqref{ax7}, by \eqref{xinl4}, \eqref{zz1}, \eqref{bcgs1} and \eqref{xuv1}, one has
\be\la{xinl6}\ba
&\sup_{  0\le t\le T   }\sigma\|\rho^{1/2}\dot{u}\|_{L^{2}}^{2}+\int_0^{T}\sigma\|\nabla\dot{u}\|_{L^{2}}^{2}dt\\
& \leq C\int_0^{\sigma(T)}(\|\rho^{\frac{1}{2}}\dot{u}\|_{L^2}^2+\|\nabla u\|_{L^2}^2+\|\nabla u\|_{L^2}^4)dt+C\int_0^{T}\sigma\|\nabla u\|_{L^4}^4dt \\
&\quad+C\int_0^{T}\sigma\|\rho^{\frac{1}{2}}\dot{u}\|_{L^2}^2\|\nabla u\|_{L^2}^2dt+C\int_0^{T}\sigma(\|\nabla u\|_{L^2}^2+\|\nabla u\|_{L^2}^6)dt\\
&\quad+C\sigma(\|\nabla u\|_{L^2}^2+\|\nabla u\|_{L^2}^4)+C\int_0^{T}\sigma\|\nabla u\|_{L^2}^4\|\rho^{\frac{1}{2}}\dot{u}\|_{L^2}^2dt\\
&\le C(\hat{\rho},M)+C\int_0^{T}\sigma\|\nabla u\|_{L^4}^4dt\\
&\le C(\hat{\rho},M)+C\int_0^{T}\sigma\|\rho^{1/2}\dot{u}\|_{L^2}^3(\|\nabla u\|_{L^2}+\|P-\bar{P}\|_{L^{2}})dt+C\int_0^{T}\sigma\|\nabla u\|_{L^2}^4dt\\
&\quad+C\int_0^{T}\sigma(\|P-\bar{P}\|_{L^{2}}^4+\|P-\bar{P}\|_{L^{4}}^4)dt\\
&\le C(\hat{\rho},M)+C\int_0^T(\sigma^{1/2}\|\nabla u\|_{L^2})(\sigma^{1/2}\|\rho^{1/2}\dot{u}\|_{L^2})\|\rho^{1/2}\dot{u}\|_{L^2}^2dt\\
&\le C(\hat{\rho},M)+C\sup_{  0\le t\le T   }\sigma^{1/2}\|\rho^{1/2}\dot{u}\|_{L^2},
\ea\ee
then \eqref{xinl5} follows from \eqref{xinl6} and Young's inequality.
\end{proof}

Now we can prove the density have a uniform (in time) upper bound, which is the key to obtain all the higher
order estimates and thus to extend the classical solution globally.
We will adopt an approach motivated by the work of \cite{hlx1,lx}.

\begin{lemma}\la{le7}
There exists a positive constant
   $\ve$
    depending    on  $\mu$,  $\lambda$, $\ga$, $a$, $\hat{\rho}$, $\Omega,$ and $M$  such that,
    if  $(\n,u)$ is a smooth solution  of
   \eqref{a1}-\eqref{ch1}     on $\O \times (0,T] $
   satisfying \eqref{zz1} and the initial data condition \eqref{dt2}, then
      \be\la{lv102}\sup_{0\le t\le T}\|\n(t)\|_{L^\infty}  \le
\frac{7\hat \rho }{4}  ,\ee
      provided $C_0\le \ve.$ Moreover, if $C_0\le \ve$, then there exists some positive constant $\tilde{C}(T)$ depending only on $T,$ $\mu,$ $\lambda,$ $\ga,$  $a,$ $\hat\rho,$  $s,$ $\Omega$, $M$ and $\vartheta$ such that for $(x,t)\in\Omega\times(0,T)$
\be\la{le7xz}\rho(x,t)\geq\tilde{C}(T)\inf_{x\in\Omega}\rho_0(x).\ee
\end{lemma}

\begin{proof}
  First, $(\ref{a1})_1$ can be rewritten as
 \be \la{z.3} D_t \n=g(\n)+b'(t), \ee where \bnn
D_t\n\triangleq\n_t+u \cdot\nabla \n ,\quad
g(\rho)\triangleq-\frac{\rho(P-\bar{P})}{\lambda(\rho)+2\mu},
\quad b(t)\triangleq- \int_0^t\frac{1}{\lambda(\rho)+2\mu}\rho Fdt. \enn

For $t\in[0,\sigma(T)],$ one deduces from \eqref{g2}, \eqref{h19}, \eqref{hh20}, \eqref{tb90}, \eqref{zz1} and Lemmas \ref{le2}, \ref{th00} that for all $0\leq t_1\leq t_2\leq\sigma(T)$,
\be \la{xbh19} \ba
&|b(t_2)-b(t_1)|\\
&\le C\int_0^{\sigma(T)}\|(\rho F)(\cdot,t)\|_{L^{\infty}}dt\\
& \le C(\hat{\rho})\int_0^{\sigma(T)}\|F\|_{L^{2}}^{1/4}\|\nabla F\|_{L^{6}}^{3/4}dt+C\int_0^{\sigma(T)}\|F\|_{L^{2}}dt\\
& \le C(\hat{\rho})\int_0^{\sigma(T)}(\|\nabla u\|_{L^2}^{1/4}+\|P-\bar{P}\|_{L^2}^{1/4})(\|\nabla\dot{u}\|_{L^2}^{3/4}+\|\nabla u\|_{L^2}^{3/2})dt+CC_0^{1/4}\\
& \le C(\hat{\rho})\int_0^{\sigma(T)}\left(\sigma^{-1/2}(\sigma^{1/2}\|\nabla u\|_{L^2})^{1/4}+(\sigma\|P-\bar{P}\|_{L^2}^2)^{1/8}\sigma^{-1/2}\right)(\sigma\|\nabla\dot{u}\|_{L^2}^{2})^{3/8}dt\\
&\quad+C\int_0^{\sigma(T)}\|\nabla u\|_{L^2}^{7/4}dt+C\int_0^{\sigma(T)}\|P-\bar{P}\|_{L^2}^{1/4}\|\nabla u\|_{L^2}^{3/2}dt+CC_0^{1/4}\\
& \le C(\hat{\rho})C_0^{1/24}\int_0^{\sigma(T)}\sigma^{-1/2}(\sigma\|\nabla\dot{u}\|_{L^2}^{2})^{3/8}dt+C\left(\int_0^{\sigma(T)}\|\nabla u\|_{L^2}^2dt\right)^{7/8}+CC_0^{1/16}\\
& \le C(\hat{\rho})C_0^{1/24}\left(\int_0^1\sigma^{-4/5}dt\right)^{5/8}\left(\int_0^{\sigma(T)}\sigma\|\nabla\dot{u}\|_{L^2}^{2}dt\right)^{3/8}+CC_0^{1/16}\\
& \le C(\hat{\rho}, M)C_0^{1/24},
\ea\ee
provide $C_0\leq\varepsilon_1$.

Combining \eqref{xbh19} with \eqref{z.3} and choosing $N_1=0$, $N_0=C(\hat{\rho}, M)C_0^{1/24}$, $\zeta_0=\hat{\rho}$
in Lemma  \ref{le1} give
   \be\la{a103}\sup_{t\in
[0,\si(T)]}\|\n\|_{L^\infty} \le \hat{\rho}
+C(\hat{\rho}, M)C_0^{1/24}\le\frac{3 \hat{\rho}  }{2},\ee
 provided $C_0\leq\varepsilon_2\triangleq\min\left\{\varepsilon_1, \left(\frac{\hat{\rho}}{2C(\hat{\rho}, M)}\right)^{24}\right\}. $

On the other hand, for $t\in[\sigma(T),T],\,\,\sigma(T)\le t_1\le t_2\le T ,$ it follows from \eqref{h19}, \eqref{tb90}, \eqref{zz1} and Lemmas \ref{le2}, \ref{th00} that
\be \la{xbh20} \ba
&|b(t_2)-b(t_1)|  \\
&\le C(\hat{\rho})\int_{t_1}^{t_2}\|F\|_{L^{\infty}}dt \\
&\le \frac{a\hat{\rho}^{\gamma+1}}{2(\lambda(\frac{3\hat{\rho}}{2})+2\mu)}(t_2-t_1)+C\int_{t_1}^{t_2}\|F\|_{L^{\infty}}^{8/3}dt \\
&\le \frac{a\hat{\rho}^{\gamma+1}}{2(\lambda(\frac{3\hat{\rho}}{2})+2\mu)}(t_2-t_1)+C\int_{t_1}^{t_2}(\|F\|_{L^2}^{2/3}\|\nabla F\|_{L^6}^{2}+\|F\|_{L^2}^{8/3})dt \\
& \le \frac{a\hat{\rho}^{\gamma+1}}{2(\lambda(\frac{3\hat{\rho}}{2})+2\mu)}(t_2-t_1)+CC_0^{1/9}\int_{\sigma(T)}^{T}(\|\nabla \dot{u}\|_{L^2}^{2}+\|\nabla u\|_{L^2}^{4})dt+CC_0^{4/9} \\
& \le \frac{a\hat{\rho}^{\gamma+1}}{2(\lambda(\frac{3\hat{\rho}}{2})+2\mu)}(t_2-t_1)+CC_0^{4/9}.
\ea\ee

Now we choose $N_0=CC_0^{4/9}$, $N_1=\frac{a\hat{\rho}^{\gamma+1}}{2(\lambda(\frac{3\hat{\rho}}{2})+2\mu)}$ in \eqref{a100} and set $\zeta_0= \frac{3\hat{\rho}}{2}$ in (\ref{a101}). Since for all $  \zeta \geq\zeta_0= \frac{3\hat{\rho}}{2}$,
$$  g(\zeta)=-\frac{ a\zeta^{\gamma+1}-\zeta\bar{P}}{\lambda(\zeta)+2\mu}\le -\frac{a\hat{\rho}^{\gamma+1}}{2(\lambda(\frac{3\hat{\rho}}{2})+2\mu)}= -N_1. $$
Together with \eqref{z.3} and \eqref{xbh20}, by Lemma \ref{le1}, we have
\be\la{a102} \sup_{t\in
[\si(T),T]}\|\n\|_{L^\infty}\le\frac{3\hat{\rho}}{2}+N_0\le \frac{3\hat{\rho}}{2}+CC_0^{4/9} \le
\frac{7\hat \n }{4},\ee provided
\be \la{xbh21} \ba C_0\leq\varepsilon\triangleq\min\{\varepsilon_2, (\frac{\hat\rho}{4C})^{9/4} \}.
\ea\ee
The combination of (\ref{a103}) with (\ref{a102}), we obtain \eqref{lv102}.

It remains to prove \eqref{le7xz}. If $\inf\limits_{x\in\Omega}\rho_0(x)=0,$ \eqref{le7xz} clearly holds. Assume that $\inf\limits_{x\in\Omega}\rho_0(x)>0,$
by \eqref{z.3}, We have
\bnn\ba
(\lambda+2\mu)D_t\rho+\rho(P-\bar{P})+\rho F=0.
\ea\enn
A simple computation shows
\bnn\ba
(\lambda+2\mu)D_t\rho^{-1}-\rho^{-1}(P-\bar{P}+F)=0.
\ea\enn
which yields that
\bnn\ba
D_t\rho^{-1}\leq C\rho^{-1}(|F|+1).
\ea\enn
Combining this with Gronwall's inequality, \eqref{xbh19} and \eqref{xbh20} gives \eqref{le7xz} and finishes
the proof of Lemma \ref{le7}.
\end{proof}
\section{\la{se5} A priori estimates (II): higher order estimates }
In this section, we will give some necessary higher-order a priori estimates of the smooth solution of \eqref{a1}-\eqref{ch1}, which make sure that one can extend the local solution to be a global one. We will adopt some ideas of the article \cite{jx01}.

From now on, we always assume that the initial energy $C_0$ satisfies (\ref{xbh21}), and the positive
constant $C $ may depend on \bnn  T,\,\, \|\rho g\|_{L^2},  \,\,\|\nabla u_0\|_{H^2},\,\,
    \|\n_0\|_{H^{3}},\,\,\|P(\n_0)\|_{H^{3}},\enn
 besides $\mu$, $\lambda$, $a$, $\ga$, $\hat{\rho},$ $\Omega$ and $
M,$ where $g\in L^2(\Omega)$ comes from \eqref{dt3}.
\begin{lemma}\la{xle1}
 There exists a positive constant $C,$ such that
\be\label{cxb2}
\sup_{0\le t\le T}\|\rho^{1/2}\dot{u}\|_{L^2}+\int_0^T\|\nabla\dot{u}\|_{L^2}^{2}dt\leq C,\ee
\be\label{cxb3}
\sup_{0\le t\le T}(\|\nabla\rho\|_{L^6}+\|u\|_{H^2})+\int_0^T\|\nabla u\|_{L^\infty}dt\leq C.\ee
\end{lemma}
\begin{proof}
Choosing $m=0$ in \eqref{ax401}, by \eqref{h18} and \eqref{xinl4}, we have
\be\la{cxb5} \ba
&\left(\|\rho^{\frac{1}{2}}\dot{u}\|_{L^2}^2\right)_t+\|\nabla\dot{u}\|_{L^2}^2\\
& \leq -\left(\int_{\partial\Omega}(u\cdot\nabla n\cdot u)Fds\right)_t+C(\|\rho^{\frac{1}{2}}\dot{u}\|_{L^2}^2+\|\nabla u\|_{L^4}^4)\\
& \leq -\left(\int_{\partial\Omega}(u\cdot\nabla n\cdot u)Fds\right)_t+C\|\rho^{\frac{1}{2}}\dot{u}\|_{L^2}^2(\|\rho^{\frac{1}{2}}\dot{u}\|_{L^2}^2+\|\nabla u\|_{L^2}^2+1) \\
&\quad+C(\|\nabla u\|_{L^2}^4+\|P-\bar{P}\|_{L^2}^4+\|P-\bar{P}\|_{L^4}^4)\\
& \leq -\left(\int_{\partial\Omega}(u\cdot\nabla n\cdot u)Fds\right)_t+C\|\rho^{\frac{1}{2}}\dot{u}\|_{L^2}^2(\|\rho^{\frac{1}{2}}\dot{u}\|_{L^2}^2+1)+C.
 \ea\ee
By Gronwall's inequality and the compatibility condition \eqref{dt3}, we deduce \eqref{cxb2} from \eqref{cxb5}, \eqref{Fnn1} and \eqref{xinl4}.

Next, we prove \eqref{cxb3} by following the proof of Lemma 5 in \cite{hlx}.
For $2\leq p\leq 6 ,$ $|\nabla\rho|^p$ satisfies \bnn \ba
& (|\nabla\rho|^p)_t + \text{div}(|\nabla\rho|^pu)+ (p-1)|\nabla\rho|^p\text{div}u  \\
 &+ p|\nabla\rho|^{p-2}(\nabla\rho)^{tr} \nabla u (\nabla\rho) +
p\rho|\nabla\rho|^{p-2}\nabla\rho\cdot\nabla\text{div}u = 0.\ea
\enn
Thus, taking $p=6$, by \eqref{h19}, \eqref{tb90} and \eqref{xinl4},
\be\la{cxb9}\ba
(\|\nabla\rho\|_{L^6})_t&\le C(1+\norm[L^{\infty}]{\nabla u} )\norm[L^6]{\nabla\rho} +C\|\nabla(\frac{F+P-\bar{P}}{\lambda(\rho)+2\mu})\|_{L^6}\\
&\le C(1+\norm[L^{\infty}]{\nabla u} )\norm[L^6]{\nabla\rho}+C(\|\rho\dot{u}\|_{L^6}+\|\nabla P\|_{L^6})\\
&\le C(1+\norm[L^{\infty}]{\nabla u} )\norm[L^6]{\nabla\rho}+C(\|\nabla\dot{u}\|_{L^2}+1).
 \ea\ee

We deduce from Gagliardo-Nirenberg's inequality, \eqref{tb90}, \eqref{h19} and \eqref{zh19} that
\be\la{cxb11}\ba
&\|\div u\|_{L^\infty}+\|\curl u\|_{L^\infty}\\
&\le C(\|F\|_{L^\infty}+\|P-\bar{P}\|_{L^\infty})+\|\curl u\|_{L^\infty} \\
&\le C(\|F\|_{L^2}+\|\nabla F\|_{L^6}+\|\curl u\|_{L^2}+\|\nabla \curl u\|_{L^6}+\|P-\bar{P}\|_{L^\infty}) \\
&\le C(\|\nabla u\|_{L^2}+\|P-\bar{P}\|_{L^2}+\|\rho\dot{u}\|_{L^6}+\|P-\bar{P}\|_{L^\infty}) \\
&\le C(\|\nabla\dot{u}\|_{L^2}+1),
\ea\ee
By Lemmas \ref{le9}, \ref{crle1}, \eqref{tb90}, \eqref{h19}, \eqref{zh19} and \eqref{h20}, it indicates that
\be\la{cxb12}\ba
\|\na u\|_{L^\infty } &\le C\left(\|{\rm div}u\|_{L^\infty }+
\|\curl u\|_{L^\infty } \right)\ln(e+\|\na^2 u\|_{L^6 }) +C\|\na
u\|_{L^2} +C \\
&\le C(1+\|\nabla\dot{u}\|_{L^2})\ln(e+\|\nabla\dot u\|_{L^2 } +\|\na \rho\|_{L^6}) \\
&\le C(1+\|\nabla\dot{u}\|_{L^2})\left[\ln(e+\|\nabla\dot u\|_{L^2 }) +\ln(e+\|\nabla\rho\|_{L^6})\right] \\
&\le C(1+\|\nabla\dot{u}\|_{L^2}^2)+C(1+\|\nabla\dot{u}\|_{L^2})\ln(e+\|\nabla\rho\|_{L^6}) .
\ea\ee
Consequently, \eqref{cxb9} becomes
\be\la{cxb13}\ba
&(e+\|\nabla\rho\|_{L^6})_t\\
&\leq C[1+\|\nabla\dot{u}\|_{L^2}^2+(1+\|\nabla\dot{u}\|_{L^2})\ln(e+\|\nabla\rho\|_{L^6})](e+\|\nabla\rho\|_{L^6}),\\
\ea\ee
which yields
\be\la{cxb14}\ba
\left(\ln(e+\|\nabla\rho\|_{L^6})\right)_t\leq C(1+\|\nabla\dot{u}\|_{L^2}^{2})+C(1+\|\nabla\dot{u}\|_{L^2})\ln(e+\|\nabla\rho\|_{L^6}).
\ea\ee
By Gronwall's inequality and \eqref{cxb2}, we obtain
\be\la{cxb15}\ba
\sup_{0\leq t\leq T}\|\nabla\rho\|_{L^{6}}\leq C .
\ea\ee

By \eqref{cxb12}, \eqref{cxb2}, \eqref{cxb15} and Lemma \ref{crle1}, we have
\be\la{cxb16}\ba
\int_0^T\|\nabla u\|_{L^\infty}dt\leq C,\,\, and\,\,\sup_{0\leq t\leq T}\|u\|_{H^{2}}\leq C.
\ea\ee

This completes the proof of Lemma \ref{xle1}.
\end{proof}

\begin{lemma}\la{xle2}
 There exists a positive constant $C$ such that
\be\la{cxb17}\ba
\sup_{0\le t\le T}\|\rho^{1/2}u_t\|_{L^2}^2 + \ia\int|\nabla u_t|^2dxdt\le C,
\ea\ee
\be\la{cxb18}\ba
\sup_{0\le t\le T}(\|{\rho- \bar{\rho}}\|_{H^2} +
 \|{P- \bar{P}}\|_{H^2}+\|\lambda(\rho)\|_{H^2})\le C.
\ea\ee
\end{lemma}
\begin{proof} By Lemma \ref{xle1}, it is easy to get
\be\la{cxb19}\ba
\|\rho^{\frac{1}{2}}
u_t\|_{L^2}^2 &\le \|\rho^{1/2}\dot u \|_{L^2}^2+\|\rho^{1/2}u\cdot\nabla u\|_{L^2}^2\\
&\le C+C\|u\|_{L^6}^2\|\nabla u\|_{L^4}^2 \\
&\le C+C\|\nabla u\|_{L^2}^2\|\nabla u\|_{H^1}^2 \\
&\le C ,
\ea\ee
and
\be \la{cxb20}\ba
 \int_0^T\|\nabla u_t\|_{L^2}^2dt &\le\int_0^T\|\nabla \dot
u\|_{L^2}^2dt + \int_0^T\|\nabla(u\cdot\nabla u)\|_{L^2}^2dt \\
&\le C+\int_0^T(\|\nabla u\|_{L^4}^4+\|u\|_{L^\infty}^2\|\nabla^{2}u\|_{L^2}^2)dt  \\
&\le C+C\int_0^T(\|\nabla^{2}u\|_{L^2}^4+\|\nabla u\|_{H^1}^{2}\|u\|_{H^2}^2)dt \\
&\le C,
\ea\ee
 so we get \eqref{cxb17} . It remains to prove \eqref{cxb18} .

 Note that $P $ satisfies \eqref{Pu2}, that is,
 $$P_t+u\cdot\nabla P+\ga P{\rm div}u=0,$$
together with $(\ref{a1})_1$, by Lemmas \ref{crle1} and \ref{xle1}, a simple computation shows that
\be \la{cxb22}\ba
&\frac{d}{dt}\left(\|\nabla^2P\|_{L^2}^2 +\|\nabla^2\rho\|_{L^2}^2\right)\\
&\le C(1+\|\nabla u\|_{L^\infty})(\|\nabla^2P\|_{L^2}^2 +\|\nabla^2\rho
\|_{L^2}^2) + C\|\nabla\dot{u}\|_{L^2}^2 + C.
\ea\ee
Consequently, by Gronwall's inequality, we deduce from \eqref{cxb22} and Lemma \ref{xle1} that
  \bnn \sup_{0\le t\le T} {\left(\|\nabla^2P\|_{L^2}^2
+\|\nabla^2\rho\|_{L^2}^2 \right)}\le C. \enn
One can handle $\lambda(\rho)$ similarly. Thus the proof of Lemma \ref{xle2} is completed.
\end{proof}
\begin{lemma}\la{xle3}
There exists a positive constant $C,$ such that
 \be\la{cxb24}
   \sup\limits_{0\le t\le T}\left(
   \|\n_t\|_{H^1}+\|P_t\|_{H^1}\right)
    + \int_0^T\left(\|\n_{tt}\|_{L^2}^2+\|P_{tt}\|_{L^2}^2\right)dt
\le C,
  \ee
\be\la{cxb25}
\sup\limits_{0\le t\le T} \|\nabla u_t\|_{L^2}^2
    + \int_0^T\|\rho^{1/2}u_{tt}\|_{L^2}^2dt
\le C.
  \ee
\end{lemma}
\begin{proof} It follows from (\ref{Pu2}) and Lemma \ref{xle1} that
\be \la{cxb26}
\|P_t\|_{L^2}\le
C\|u\|_{L^\infty}\|\nabla P\|_{L^2}+C\|\nabla u\|_{L^2}\le C.
\ee
Differentiating (\ref{Pu2}) with respect to $x$ yields
\bnn
\nabla P_t+u\cdot\nabla\nabla P+\nabla u\cdot\nabla P+\ga \nabla P {\rm div}u+\ga P  \nabla{\rm div}u=0.
\enn
By Lemmas \ref{xle1} and \ref{xle2}, we obtain
\bn\la{cxb27} \|\nabla P_t\|_{L^2}\le C\|u\|_{L^\infty}\|\nabla^2
P\|_{L^2}+C\|\nabla u\|_{L^6}\|\nabla P\|_{L^3}+C\|\nabla^2
u\|_{L^2}\le C.\en
The combination of (\ref{cxb26}) with (\ref{cxb27})
implies
\bn \la{cxb28}\sup_{0\le t\le T}\|P_t\|_{H^1}\le C.
\en

By \eqref{Pu2} again, it is easy to check that $P_{tt}$ satisfies
\be\la{cxb29} P_{tt} + \gamma P_t{\rm div}u +
\gamma P{\rm div}u_t + u_t\cdot\nabla P + u\cdot\nabla P_t = 0.
\ee
Multiplying \eqref{cxb29} by $P_{tt}$ and integrating over $\Omega\times[0,T],$ by \eqref{cxb28} and Lemmas \ref{xle1}, \ref{xle2}, we get
\bnn \ba
&\int_0^T\|P_{tt}\|_{L^2}^2dt \\
& = -\int_0^T\int\gamma P_{tt}P_t\div udxdt - \int_0^T\int\gamma P_{tt}P\div u_tdxdt  \\
& \quad - \int_0^T\int P_{tt}u_t\cdot\nabla Pdxdt - \int_0^T\int P_{tt}u\cdot\nabla P_tdxdt \\
& \le C\int_0^T\|P_{tt}\|_{L^2}(\|P_{t}\|_{L^3}\|\nabla u\|_{L^6}+\|\nabla u_t\|_{L^2}+\|u_t\|_{L^3}\|\nabla P\|_{L^6}+\|u\|_{L^\infty}\|\nabla P_{t}\|_{L^2})dt\\  & \le C\int_0^T\|P_{tt}\|_{L^2}(1+\|\nabla u_t\|_{L^2})dt\\
& \le \frac{1}{2}\int_0^T\|P_{tt}\|_{L^2}^2dt+C\int_0^T\|\nabla u_t\|_{L^2}^2dt+C\\
& \le \frac{1}{2}\int_0^T\|P_{tt}\|_{L^2}^2dt+C,
\ea\enn
so we have
$$\int_0^T\|P_{tt}\|_{L^2}^2dt \le C .$$

One can handle $\n_t$ and
$\n_{tt}$ similarly. Thus (\ref{cxb24}) is proved.

It remains to prove (\ref{cxb25}). Introducing the function
$$H(t)=\int(\lambda(\rho)+2\mu)(\div u_t)^{2}dx+\mu\int|\curl u_t|^{2}dx .$$
Since $u_t\cdot n = 0$ on $\partial\Omega$, by Lemma \ref{crle1}, we have
\be\ba\la{cxb40}
\|\nabla u_t\|_{L^2}^2\leq CH(t).
\ea\ee

Differentiating  $(\ref{a1})_2$  with respect to $t$ shows
\be\la{zj1} \ba
&\rho u_{tt}+\rho u\cdot\nabla u_t-\nabla((\lambda(\rho)+2\mu)\div u_t)+\mu\nabla\times\curl u_t\\
&=-\rho_t\dot{u}-\rho u_t\cdot\nabla u+\nabla(\lambda_t\div u)-\nabla P_t.
\ea\ee
Multiplying \eqref{zj1} by $u_{tt}$ and integrating over $\Omega,$ we obtain
\be\la{cxb34} \ba
&\frac{1}{2}\frac{d}{dt}H(t)+\int\rho|u_{tt}|^2dx  \\
&=-\int\left(2\rho u\cdot\nabla u_t\cdot u_{tt}+\rho u_t\cdot\nabla u\cdot u_{tt}\right)dx-\int\rho u\cdot\nabla(u\cdot\nabla u)\cdot u_{tt}dx \\
&\quad -\int\rho u\cdot\nabla u_{tt}\cdot\dot{u}dx+\frac{1}{2}\int\lambda_t(\div u_t)^{2}dx - \int\lambda_t\div u\div u_{tt}dx+\int P_t\div u_{tt}dx.
\ea \ee

The terms on the right-hand side of \eqref{cxb34} will be estimated one by one.

First, it follows from \eqref{cxb3}, \eqref{g1}, and Sobolev embedding theorem that
\be\la{zj2} \ba
&\left|\int\left(2\rho u\cdot\nabla u_t\cdot u_{tt}+\rho u_t\cdot\nabla u\cdot u_{tt}\right)dx\right|+\left|\int\rho u\cdot\nabla(u\cdot\nabla u)\cdot u_{tt}dx\right| \\
&\leq\varepsilon\int\rho|u_{tt}|^2dx+C(\hat{\rho})(\|u\|_{L^{\infty}}^2\|\nabla u_t\|_{L^2}^2+\|\rho^{1/2}u_t\|_{L^4}^2\|\nabla u\|_{L^4}^2)\\
&\quad+C\int(\rho|u|^2|\nabla u|^4+\rho|u|^4|\nabla^2u|^2)dx\\
&\leq\varepsilon\int\rho|u_{tt}|^2dx+C(\|\nabla u_t\|_{L^2}^2+1).
\ea \ee
Next, a direct calculation gives
\be\la{zj3} \ba
&-\int\rho u\cdot\nabla u_{tt}\cdot\dot{u}dx\\
&=-\frac{d}{dt}\int\rho u\cdot\nabla u_{t}\cdot\dot{u}dx+\int(\rho u)_t\cdot\nabla u_t\cdot\dot{u}dx\\
&\quad+\int\rho u\cdot\nabla u_t\cdot(u_{tt}+u_t\cdot\nabla u+u\cdot\nabla u_t)dx.
\ea \ee
On the one hand, it follows from \eqref{g1}, \eqref{tb90}, \eqref{cxb3} and \eqref{cxb24} that
\be\la{zj4} \ba
&\left|\int(\rho u)_t\cdot\nabla u_t\cdot\dot{u}dx\right|\\
&\leq\int|\rho u_t\cdot\nabla u_{t}\cdot\dot{u}|dx+\int|\rho_t u\cdot\nabla u_t\cdot\dot{u}|dx\\
&\leq C(\hat{\rho})\|u_t\|_{L^6}\|\nabla u_t\|_{L^2}\|\dot{u}\|_{L^6}+\|\rho_t\|_{L^2}\|\nabla u_t\|_{L^6}\|\dot{u}\|_{L^6}\\
&\leq C(\hat{\rho})\|\nabla u_t\|_{L^2}^2(\|\nabla u_t\|_{L^2}+1)+C(\|\nabla u_t\|_{L^2}+\|\nabla^2 u_t\|_{L^2})(\|\nabla u_t\|_{L^2}+1)\\
&\leq C(\hat{\rho})\|\nabla u_t\|_{L^2}^4+\delta\|\nabla^2 u_t\|_{L^2}^2+C.
\ea \ee
On the other hand, Cauchy's inequality, \eqref{g1} and \eqref{cxb3} show that
\be\la{zj5} \ba
&\left|\int\rho u\cdot\nabla u_t\cdot(u_{tt}+u_t\cdot\nabla u+u\cdot\nabla u_t)dx\right|\\
&\leq\varepsilon\int\rho|u_{tt}|^2dx+C(\hat{\rho})\|u\|_{L^\infty}^2\int|\nabla u_t|^2dx+C(\hat{\rho})\|u\|_{L^\infty}\|\nabla u_t\|_{L^2}\|u_t\|_{L^4}\|\nabla u\|_{L^4}\\
&\quad+C(\hat{\rho})\|u\|_{L^\infty}^2\|\nabla u_t\|_{L^2}^2\\
&\le \varepsilon\int\rho|u_{tt}|^2dx+C\int|\nabla u_t|^2dx.
\ea \ee

Putting \eqref{zj4} and \eqref{zj5} into \eqref{zj3} shows
\be\la{zj6} \ba
&-\int\rho u\cdot\nabla u_{tt}\cdot\dot{u}dx\\
&=-\frac{d}{dt}\int\rho u\cdot\nabla u_{t}\cdot\dot{u}dx+\varepsilon\int\rho|u_{tt}|^2dx+\delta\|\nabla^2 u_t\|_{L^2}^2+C\|\nabla u_t\|_{L^2}^4+C.
\ea \ee

Note that
\be\la{zjh1} \ba
\lambda_t=(1-\beta)\lambda\div u-\div(\lambda u).
\ea \ee
Since $\|\nabla\lambda\|_{L^6}$ is bounded,
therefore \be\la{hzj6} \ba\|\lambda_t\|_{L^6}\leq C(\hat{\rho})\|\nabla u\|_{L^6}+C(\hat{\rho})\|u\|_{L^\infty}\|\nabla\lambda\|_{L^6}\leq C,\ea \ee
by Sobolev inequality, one has
\be\la{zj7} \ba
&\int\lambda_t(\div u_t)^2dx\leq C\|\lambda_t\|_{L^2}\|\nabla u_t\|_{L^4}^2\leq \frac{\delta}{3}\|\nabla^2 u_t\|_{L^2}^2+C\|\nabla u_t\|_{L^2}^2.
\ea \ee
So we have
\be\la{zj8} \ba
&-\int\lambda_t\div u\div u_{tt}dx\\
&=-\frac{d}{dt}\int\lambda_t\div u\div u_tdx+\int\lambda_t(\div u_t)^2dx\\
&\quad-(\beta-1)\int(\lambda\div u)_t\div u\div u_tdx+\int(\lambda u)_t\cdot\nabla(\div u\div u_t)dx.
\ea \ee
It follows from Cauchy's inequality and \eqref{cxb3} that
\be\la{zj9} \ba
&\left|\int(\lambda\div u)_t\div u\div u_tdx\right|\\
&\leq C(\|\lambda_t\|_{L^2}\|\nabla u\|_{L^6}^2\|\nabla u_t\|_{L^6}+\|\lambda\|_{L^\infty}\|\nabla u\|_{L^2}\|\nabla u_t\|_{L^4}^2)\\
&\leq\frac{\delta}{3}\|\nabla^2 u_t\|_{L^2}^2+C\|\nabla u_t\|_{L^2}^2+C,
\ea \ee
and that
\be\la{zj10} \ba
&\left|\int(\lambda u)_t\cdot\nabla(\div u\div u_t)dx\right|\\
&\leq\int(|\lambda_t||u||\nabla u||\nabla^2u_t|+|\lambda_t||u||\nabla^2 u||\nabla u_t|+|\lambda||u_t||\nabla^2 u||\nabla u_t|+|\lambda||u_t||\nabla u||\nabla^2u_t|)dx\\
&\leq C(\|\lambda_t\|_{L^6}\|u\|_{L^\infty}\|\nabla u\|_{L^3}\|\nabla^2u_t\|_{L^2}+\|\lambda_t\|_{L^6}\|u\|_{L^\infty}\|\nabla^2 u\|_{L^2}\|\nabla u_t\|_{L^3})\\
&\quad+C(\|u_t\|_{L^6}\|\nabla^2u\|_{L^2}\|\nabla u_t\|_{L^3}+\|u_t\|_{L^6}\|\nabla u\|_{L^3}\|\nabla^2 u_t\|_{L^2})\\
&\leq\frac{\delta}{3}\|\nabla^2 u_t\|_{L^2}^2+C(1+\|\nabla u_t\|_{L^2}^2).
\ea \ee
Putting \eqref{zj7}, \eqref{zj9}, and \eqref{zj10} into \eqref{zj8} gives
\be\la{zj11} \ba
&-\int\lambda_t\div u\div u_{tt}dx\leq-\frac{d}{dt}\int\lambda_t\div u\div u_tdx+\delta\|\nabla^2 u_t\|_{L^2}^2+C(1+\|\nabla u_t\|_{L^2}^2).
\ea \ee

Finally, it follows from \eqref{Pu1}, \eqref{cxb18}, \eqref{cxb24} and \eqref{g1} that
\be\la{zj12} \ba
\int P_t\div u_{tt}dx&=\frac{d}{dt}\int P_t\div u_tdx-\int(Pu)_t\cdot\nabla\div u_tdx+(\gamma-1)\int(P\div u)_t\div u_tdx\\
&\leq\frac{d}{dt}\int P_t\div u_tdx+C(\|u\|_{L^\infty}\|P_t\|_{L^2}+\|P\|_{L^\infty}\|u_t\|_{L^2})\|\nabla^2u_t\|_{L^2}\\
&\quad+C(\|P_t\|_{L^2}\|\nabla u\|_{L^3}\|\nabla u_t\|_{L^6}+\|P\|_{L^\infty}\|\nabla u_t\|_{L^2}^2)\\
&\leq\frac{d}{dt}\int P_t\div u_tdx+\delta\|\nabla^2u_t\|_{L^2}^2+C(\|\nabla u_t\|_{L^2}^2+1).
\ea \ee

Substituting \eqref{zj2}, \eqref{zj6}, \eqref{zj7}, \eqref{zj11}, \eqref{zj12} into \eqref{cxb34} and choosing $\varepsilon$ suitably small lead to
\be\la{zj13} \ba
\psi'(t)+\int\rho|u_{tt}|^2dx\leq C\delta\|\nabla^2u_t\|_{L^2}^2+C\|\nabla u_t\|_{L^2}^4+C,
\ea \ee
where
\bnn\ba
\psi(t)&\triangleq\int((\lambda(\rho)+2\mu)(\div u_t)^2+\mu|\curl u_t|^2)dx\\
&\quad-2\int(P_t\div u_t-\lambda_t\div u\div u_t-\rho u\cdot\nabla u_t\cdot\dot{u})dx
\ea\enn
satisfies
\bnn\ba
C_0(\mu)\|\nabla u_t\|_{L^2}^2-C\leq\psi(t)\leq C\|\nabla u_t\|_{L^2}^2+C,
\ea\enn
duo to the following simple fact:
\bnn\ba
&\left|\int(P_t\div u_t-\lambda_t\div u\div u_t-\rho u\cdot\nabla u_t\cdot\dot{u})dx\right|\\
&\leq C(\|P_t\|_{L^2}\|\nabla u_t\|_{L^2}+\|\lambda_t\|_{L^6}\|\nabla u\|_{L^3}\|\nabla u_t\|_{L^2})\\
&\quad+C\|u\|_{L^\infty}\|\nabla u_t\|_{L^2}(\|\rho^{1/2} u_t\|_{L^2}+\|u\cdot\nabla u\|_{L^2})\\
&\leq\varepsilon\|\nabla u_t\|_{L^2}^2+C,
\ea\enn
which comes from \eqref{cxb3}, \eqref{cxb17}, \eqref{cxb24} and \eqref{hzj6}.

Then, it remains to estimate the first term on the right-hand side of \eqref{zj13}. In fact, we obtain from \eqref{zjh1} that
\be\la{zj14} \ba
\|\nabla\lambda_t\|_{L^2}&\leq C(\|\nabla u\|_{L^\infty}\|\nabla\lambda\|_{L^2}+\|u\|_{L^\infty}\|\nabla^2\lambda\|_{L^2}+\|\nabla^2u\|_{L^2})\\
&\leq C+C\|\nabla^3u\|_{L^2}\\
&\leq C+C\|\nabla u_t\|_{L^2},
\ea \ee
where in the last inequality we have used Lemma \ref{crle1}, \eqref{cxb2}, \eqref{cxb3}, \eqref{cxb18} and the following facts:
\be\la{zjh2} \ba
\|\nabla(\rho\dot{u})\|_{L^2}&\leq\||\nabla\rho||u_t|\|_{L^2}+\|\rho\nabla u_t\|_{L^2}+\||\nabla\rho||u||\nabla u|\|_{L^2}\\
&\quad+\|\rho|\nabla u|^2\|_{L^2}+\|\rho|u||\nabla^2u|\|_{L^2}\\
&\leq\|\nabla\rho\|_{L^3}\|u_t\|_{L^6}+C\|\nabla u_t\|_{L^2}+C\|\nabla\rho\|_{L^3}\|u\|_{L^\infty}\|\nabla u\|_{L^6}\\
&\quad+C\|\nabla u\|_{L^3}\|\nabla u\|_{L^6}+C\|u\|_{L^\infty}\|\nabla^2u\|_{L^2}\\
&\leq C\|\nabla u_t\|_{L^2}+C,
\ea \ee
and
\be\la{zj20} \ba&\|\nabla^3 u\|_{L^2}\\&\le C (\|\div u\|_{H^2}+\|\curl u\|_{H^2})\\
&\le C\left(\left\|\frac{F+P-\bar{P}}{\lambda(\rho)+2\mu}\right\|_{H^2}+\|\rho\dot{u}\|_{L^2}+\|\nabla(\rho\dot{u})\|_{L^2}+\|\nabla u\|_{L^2}\right)\\
&\le C\left\||\nabla^2(\frac{1}{\lambda(\rho)+2\mu})||F+P-\bar{P}|\right\|_{L^2}+C\left\||\nabla(\frac{1}{\lambda(\rho)+2\mu})||\nabla(F+P-\bar{P})|\right\|_{L^2}\\
&\quad+C\left\|\frac{\nabla^2(F+P-\bar{P})}{\lambda(\rho)+2\mu}\right\|_{L^2}+C(\|\rho\dot{u}\|_{L^2}+\|\nabla(\rho\dot{u})\|_{L^2}+1)\\
&\le C(\hat{\rho})(\|F\|_{L^\infty}+1)\left\|\nabla^2(\frac{1}{\lambda(\rho)+2\mu})\right\|_{L^2}+C(\hat{\rho})\left\||\nabla\rho||\nabla(F+P-\bar{P})|\right\|_{L^2}\\
&\quad+C(\hat{\rho})(\|\nabla^2F\|_{L^2}+\|\nabla^2P\|_{L^2})+C(\|\rho\dot{u}\|_{L^2}+\|\nabla(\rho\dot{u})\|_{L^2}+1)\\
&\le C(\hat{\rho})(\|\nabla\dot{u}\|_{L^2}+1)(\|\nabla^2\lambda\|_{L^2}+\|\nabla\lambda\|_{L^4}^2)+C(\hat{\rho})(\|\nabla\rho\|_{L^3}\|\nabla F\|_{L^6}+\|\nabla\rho\|_{L^4}^2)\\
&\quad+C(\|\rho\dot{u}\|_{L^2}+\|\nabla(\rho\dot{u})\|_{L^2}+1)\\
&\leq C\|\nabla u_t\|_{L^2}+C,
 \ea\ee
 due to \eqref{cxb18} and Lemmas \ref{xle1} and \ref{crle1}.

Using slip boundary condition \eqref{ch1}, we obtain from \eqref{zj1} that
\bnn\ba
&\|\nabla((\lambda(\rho)+2\mu)\div u_t)\|_{L^2}^2+\mu^2\|\nabla\times\curl u_t\|_{L^2}^2\\
&=\int|\nabla((\lambda(\rho)+2\mu)\div u_t)-\mu\nabla\times\curl u_t|^2dx\\
&=\int|\rho u_{tt}+\rho_t\dot{u}+\rho u\cdot\nabla u_t+\rho u_t\cdot\nabla u-\nabla(\lambda_t\div u)+\nabla P_t|^2dx\\
&\leq C(\hat{\rho})\int\rho|u_{tt}|^2dx+C(\|\rho_t\dot{u}\|_{L^2}^2+\|\nabla u_t\|_{L^2}^2+\|u_t\|_{L^4}^2\|\nabla u\|_{L^4}^2)\\
&\quad+C(\|\nabla u\|_{L^\infty}^2\|\nabla \lambda_t\|_{L^2}^2+\|\lambda_t\|_{L^4}^2\|\nabla^2 u\|_{L^4}^2+\|\nabla P_t\|_{L^2}^2)\\
&\leq C\int\rho|u_{tt}|^2dx+C(\|\rho_t\|_{L^4}^2\|\dot u\|_{L^4}^2+\|\nabla u_t\|_{L^2}^4+1)\\
&\leq C\int\rho|u_{tt}|^2dx+C(\|\nabla u_t\|_{L^2}^4+1),
\ea\enn
where in the last inequality we have used \eqref{cxb3}, \eqref{cxb24}, \eqref{zj14}, \eqref{cxb27}, and \eqref{zjh2}. This combined with Lemmas \ref{crle1} and \ref{crle2} yields that

\bnn\ba
&\|\nabla^2 u_t\|_{L^2}\\
&\leq C(\|\div u_t\|_{H^1}+\|\curl u_t\|_{H^1})\\
&\leq C(\|\nabla u_t\|_{L^2}+\|\nabla\div u_t\|_{L^2}+\|\nabla\curl u_t\|_{L^2})\\
&\leq C(\|\nabla\div u_t\|_{L^2}+\|\nabla\times\curl u_t\|_{L^2}+\|\nabla u_t\|_{L^2})\\
&\leq C(\|\nabla((\lambda(\rho)+2\mu)\div u_t)\|_{L^2}+\|\nabla\times\curl u_t\|_{L^2})\\
&\quad+C(\|\nabla\lambda(\rho)\|_{L^4}\|\nabla u_t\|_{L^4}+C\|\nabla u_t\|_{L^2})\\
&\leq C(\|\rho^{1/2}u_{tt}\|_{L^2}+\|\nabla u_t\|_{L^2}^2+1)+\frac{1}{2}\|\nabla^2u_t\|_{L^2},
\ea \enn
Consequently,
\be\la{zjh3} \ba
\|\nabla^2 u_t\|_{L^2}\leq C(\|\rho^{1/2}u_{tt}\|_{L^2}+\|\nabla u_t\|_{L^2}^2+1).
\ea \ee
Substituting \eqref{zjh3} into \eqref{zj13} and choosing $\delta$ suitably small give \eqref{cxb25}.
This completes the proof of Lemma \ref{xle3}.
\end{proof}
\begin{lemma}\la{xle4}
It holds that
\be\la{cxb44}\ba \sup_{t\in[0,T]}\left(\|\rho- \bar{\rho}\|_{H^{3}} +\|P-\bar{P}\|_{H^{3}}+\|\nabla^3\lambda(\rho)\|_{L^{2}}\right)\le C,\ea \ee
\be\la{cxb45}\ba\sup_{t\in[0,T]}\|u\|_{H^3}^2 +\int_0^{T}(\|\nabla u\|_{H^3}^2+\|\nabla u_t\|_{H^1}^2)dt\le C.\ea \ee
\end{lemma}
 \begin{proof}
 Let's start with \eqref{cxb45}. \eqref{cxb2}, \eqref{zjh2} and \eqref{cxb25} show that
\be\la{zj19} \ba
\sup\limits_{0\le t\le T} \|\rho\dot{u}\|_{H^1}\le C.
\ea \ee
As a consequence of \eqref{cxb3}, \eqref{zj20} and \eqref{dt1}, one has
\be\la{zj21} \ba
\sup\limits_{0\le t\le T}\|u\|_{H^3}\le C.
 \ea\ee
Together with \eqref{cxb25} and \eqref{zjh3} yields
\be\la{cxb50}\ba
\int_0^T\|\nabla u_t\|_{H^1}^2dt\leq C.
\ea \ee
By Lemma \ref{crle1}, we get
\be\la{cxb53}\ba     &\|\nabla^4u\|_{L^2}\\
&\le C(\|\div u\|_{H^3}+\|\curl u\|_{H^3})\\
&\le C\left(\left\|\frac{F+P-\bar{P}}{\lambda(\rho)+2\mu}\right\|_{H^3}+\|\rho\dot{u}\|_{L^2}+\|\nabla(\rho\dot{u})\|_{L^2}+\|\nabla^2(\rho\dot{u})\|_{L^2}+\|\nabla u\|_{L^2}\right)\\
&\le C\left(\left\|\frac{F+P-\bar{P}}{\lambda(\rho)+2\mu}\right\|_{H^2}+\left\|\nabla^3\left(\frac{F+P-\bar{P}}{\lambda(\rho)+2\mu}\right)\right\|_{L^2}+\|\nabla^2(\rho\dot{u})\|_{L^2}+\|\nabla u_t\|_{L^2}+1\right)\\
&\le C(\|\nabla u_t\|_{H^1}+\|\nabla^3 \lambda\|_{L^2}+\|\nabla^3 u\|_{L^2}+\|\nabla^3 P\|_{L^2}+1)\\
&\le C(\|\nabla u_t\|_{H^1}+\|\nabla^3 \lambda\|_{L^2}+\|\nabla^3 P\|_{L^2}+1),
\ea \ee
where one has used \eqref{zj19}, \eqref{zj21} and the following simple facts:
\bnn\ba
\|\nabla^2(\rho u_t)\|_{L^2}&\leq C(\||\nabla^2\rho||u_t|\|_{L^2}+\||\nabla\rho||\nabla u_t|\|_{L^2}+\|\nabla^2 u_t\|_{L^2})\\
&\leq C(\|\nabla^2\rho\|_{L^2}\|\nabla u_t\|_{H^1}+\|\nabla\rho\|_{L^3}\|\nabla u_t\|_{L^6}+\|\nabla^2 u_t\|_{L^2})\\
&\leq C(1+\|\nabla u_t\|_{H^1}),
\ea\enn
and
\bnn\ba
\|\nabla^2(\rho u\cdot\nabla u)\|_{L^2}&\leq C(\||\nabla^2(\rho u)||\nabla u|\|_{L^2}+\||\nabla(\rho u)||\nabla^2 u|\|_{L^2}+\|\nabla^3 u\|_{L^2})\\
&\leq C(1+\|\nabla^2(\rho u)\|_{L^2}\|\nabla u\|_{H^2}+\|\nabla(\rho u)\|_{L^3}\|\nabla^2 u\|_{L^6}+\|\nabla^3 u\|_{L^2})\\
&\leq C(1+\|u\|_{L^\infty}\|\nabla^2\rho\|_{L^2}+\|\nabla\rho\|_{L^6}\|\nabla u\|_{L^3}+\|\nabla^2u\|_{L^2}+\|\nabla^3 u\|_{L^2})\\
&\leq C,
\ea\enn
due to \eqref{cxb18} and \eqref{zj21}. By using $\eqref{a1}_1$, \eqref{zj21}, \eqref{cxb18}, and \eqref{cxb53}, one may get that
\bnn\ba
&(\|\nabla^3P\|_{L^2}^2)_t\\&\leq C(\||\nabla^3u||\nabla P|\|_{L^2}+\||\nabla^2u||\nabla^2 P|\|_{L^2}+\||\nabla u||\nabla^3 P|\|_{L^2}+\|\nabla^4 u\|_{L^2})\|\nabla^3 P\|_{L^2}\\
&\leq C(\|\nabla^3u\|_{L^2}\|\nabla P\|_{H^2}+\|\nabla^2u\|_{L^3}\|\nabla^2 P\|_{L^6}+\|\nabla u\|_{L^\infty}\|\nabla^3 P\|_{L^2})\|\nabla^3 P\|_{L^2}\\
&\quad+C(1+\|\nabla^2u_t\|_{L^2}+\|\nabla^3 P\|_{L^2})\|\nabla^3 P\|_{L^2}\\
&\leq C(1+\|\nabla u_t\|_{H^1}+\|\nabla^3 \rho\|_{L^2}^2),
\ea\enn
which, together with Gronwall's inequality and \eqref{cxb50}, yields that
\be\la{zj22}\ba
\sup_{t\in[0,T]}\|\nabla^3P\|_{L^{2}}\le C.
\ea \ee
Similar to the proof of \eqref{zj22}, we get
\be\la{zj220}\ba
\sup_{t\in[0,T]}\|\nabla^3\lambda\|_{L^{2}}\le C.
\ea \ee
Collecting all these estimates \eqref{cxb50}, \eqref{cxb53}, \eqref{zj22}, \eqref{zj220} and \eqref{cxb18} shows
\be\la{zj23}\ba
\sup_{t\in[0,T]}\|P-\bar{P}\|_{H^{3}}+\int_0^{T}\|\nabla u\|_{H^3}^2dt\leq C.
\ea \ee
One can handle with $\rho-\bar{\rho}$ similarly.
Estimates \eqref{cxb45} thus follows from \eqref{zj21}, \eqref{cxb50} and \eqref{zj23}. Hence the proof of Lemma \ref{xle4} is finished.
\end{proof}
\begin{lemma}\la{xle5} For any $0<\tau<T$, there exists some positive constant $C(\tau)$ such that
\be \la{cxb58}
\sup_{\tau\le t\le T}(\|\nabla u_t\|_{H^1}
 +\|\nabla^4 u\|_{L^2})+\int_\tau^T\int|\nabla u_{tt}|^2dxdt\leq C(\tau).
 \ee
\end{lemma}
\begin{proof} Differentiating $(\ref{a1})_2$ with respect to $t$ twice, we have
 \be\la{cxb59}\ba
&\n u_{ttt}+\n u\cdot\na u_{tt}-\nabla((\lambda(\rho)+2\mu){\rm div}u_{tt})+\mu\nabla\times\curl u_{tt}\\
&= 2{\rm div}(\n u)u_{tt}
+{\rm div}(\n u)_{t}u_t-2(\n u)_t\cdot\na u_t-(\n_{tt} u+2\n_t u_t)
\cdot\na u\\& \quad- \n u_{tt}\cdot\na u-\na P_{tt}+2\nabla(\lambda_t\div u_t)+\nabla(\lambda_{tt}\div u).
 \ea\ee

Then, multiplying (\ref{cxb59}) by $u_{tt}$ and integrating over $\Omega$ lead to
\be \la{cxb60}\ba
&\frac{1}{2}\frac{d}{dt}\int_{ }\rho
|u_{tt}|^2dx+\int_{ }(\lambda(\rho)+2\mu)(\div u_{tt})^2dx+\mu\int_{ }|\curl u_{tt}|^2dx \\
&=-2\int_{ }  \lambda_t\div u_t\div u_{tt}dx-\int_{ }\lambda_{tt}\div u\div u_{tt}dx\\
&\quad -4\int_{ }\rho u\cdot\nabla u_{tt}\cdot u_{tt}dx-\int_{ }(\rho u)_t\cdot(\nabla(u_t\cdot u_{tt})+2\nabla u_t\cdot u_{tt})dx\\
&\quad -\int_{ }(\rho u)_t\cdot\nabla(u\cdot\nabla u\cdot u_{tt})dx-2\int_{ }\rho_tu_t\cdot\nabla u\cdot u_{tt}dx\\
&\quad -\int_{ }\rho u_{tt}\cdot\nabla u\cdot u_{tt}dx+\int_{ } P_{tt}\div u_{tt}dx\\
&\triangleq\sum_{i=1}^8J_i.
\ea\ee

Let us estimate $J_i$ for $i=1,2,\cdots,8$.

First, we deduce from \eqref{hzj6} that
\be \la{zj25}\ba
|J_1|&\leq C\|\lambda_t\|_{L^6}\|\nabla u_t\|_{L^3}\|\nabla u_{tt}\|_{L^2}\\
&\leq \varepsilon\|\nabla u_{tt}\|_{L^2}^2+C\|\nabla u_t\|_{H^1}^2.
\ea\ee

Next, Cauchy's inequality, \eqref{zj20} and \eqref{cxb45} give
\be \la{zj26}\ba
|J_2|&\leq \varepsilon\|\nabla u_{tt}\|_{L^2}^2+C\|\nabla u\|_{L^\infty}^2\|\lambda_{tt}\|_{L^2}^2\\
&\leq \varepsilon\|\nabla u_{tt}\|_{L^2}^2+C\|\lambda_{tt}\|_{L^2}^2.
\ea\ee
Using \eqref{zjh1}, \eqref{cxb45} and \eqref{zj14}, we estimate the last term on the right-hand side of \eqref{zj26} as follows:
\be \la{zj27}\ba
\|\lambda_{tt}\|_{L^2}&\leq C(\||u_t||\nabla\lambda|\|_{L^2}+\||u||\nabla\lambda_t|\|_{L^2}+\|\lambda_{t}\div u\|_{L^2}+\|\lambda\div u_t\|_{L^2})\\
&\leq C(\|u_t\|_{L^3}\|\nabla\lambda\|_{L^6}+\|u\|_{L^\infty}\|\nabla\lambda_t\|_{L^2}+\|\nabla u\|_{L^\infty}\|\lambda_{t}\|_{L^2}+\|\nabla u_t\|_{L^2})\\
&\leq C(\|\nabla u_t\|_{L^2}+1).
\ea\ee
Putting \eqref{zj27} into \eqref{zj26} gives
\be \la{zj28}\ba
|J_2|\leq \varepsilon\|\nabla u_{tt}\|_{L^2}^2+C(\|\nabla u_t\|_{L^2}^4+1).
\ea\ee

Next, the combination of the Cauchy's inequality with \eqref{cxb3} yields that
\be \la{zj29}\ba
|J_3|\leq C(\hat{\rho})\|u\|_{L^\infty}\|\rho^{1/2}u_{tt}\|_{L^2}\|\nabla u_{tt}\|_{L^2}\leq\varepsilon\|\nabla u_{tt}\|_{L^2}^2+C\|\rho^{1/2}u_{tt}\|_{L^2}^2.
\ea\ee
We obtain from H\"{o}lder's inequality and \eqref{g1} that
\be \la{zj30}\ba
|J_4|&\leq C\|(\rho u)_t\|_{L^6}(\|u_{tt}\|_{L^6}\|\nabla u_t\|_{L^2}+\|u_t\|_{L^6}\|\nabla u_{tt}\|_{L^2})\\
&\leq C(1+\|\nabla u_t\|_{L^2})\|\nabla u_{tt}\|_{L^2}\|\nabla u_t\|_{L^2}\\
&\leq \varepsilon\|\nabla u_{tt}\|_{L^2}^2+C(1+\|\nabla u_t\|_{L^2}^4).
\ea\ee

Then, it follows from \eqref{cxb3} and \eqref{g1} that
\be \la{zj31}\ba
|J_5|&\leq C\|(\rho u)_t\|_{L^6}(\|\nabla u\|_{L^4}^2\|u_{tt}\|_{L^6}+\|u\|_{L^\infty}\|\nabla^2 u\|_{L^2}\|u_{tt}\|_{L^6}\\
&\quad+\|u\|_{L^\infty}\|\nabla u\|_{L^3}\|\nabla u_{tt}\|_{L^2})\\
&\leq C(1+\|\nabla u_t\|_{L^2})\|\nabla u_{tt}\|_{L^2}\\
&\leq \varepsilon\|\nabla u_{tt}\|_{L^2}^2+C(1+\|\nabla u_t\|_{L^2}^2).
\ea\ee

Next, H\"{o}lder's inequality together with \eqref{g1}, \eqref{cxb3}, and \eqref{cxb24} gives
\be \la{zj32}\ba
|J_6|&\leq C\int_{}|\rho_t||u_t||\nabla u||u_{tt}|dx\\
&\leq C\|\rho_t\|_{L^6}\|u_t\|_{L^6}\|\nabla u\|_{L^2}\|u_{tt}\|_{L^6}\\
&\leq C\|\nabla u_t\|_{L^2}\|\nabla u_{tt}\|_{L^2}\\
&\leq \varepsilon\|\nabla u_{tt}\|_{L^2}^2+C\|\nabla u_t\|_{L^2}^2.
\ea\ee

Finally, similar to \eqref{zj27}, we have
\be \la{zj33}\ba
\|P_{tt}\|_{L^2}\leq C(1+\|\nabla u_t\|_{L^2}),
\ea\ee
which together with direct calculations gives
\be  \la{zj34}\ba
|J_7|+|J_8|&\le C(\hat{\rho})\|\nabla u\|_{L^\infty}\int_{}\rho|u_{tt}|^2dx+\varepsilon\int_{}|\nabla u_{tt}|^2dx+C\|P_{tt}\|_{L^2}^2\\
&\le C(\hat{\rho})\|\nabla u\|_{L^\infty}\|\rho^{1/2}u_{tt}\|_{L^2}^2+\varepsilon\|\nabla u_{tt}\|_{L^2}^2+C(1+\|\nabla u_t\|_{L^2}^2).
\ea\ee
Substituting \eqref{zj25}, \eqref{zj28}, \eqref{zj29}, \eqref{zj30}, \eqref{zj31}, \eqref{zj32}, \eqref{zj34} into \eqref{cxb60} and choosing $\varepsilon$ suitably small, one obtains by using \eqref{cxb25}, \eqref{cxb45} and Gronwall's inequality that
\be  \la{zj36}\ba
\sup_{\tau\le t\le T}\int_{}\rho|u_{tt}|^2dx+\int_{\tau}^T\int|\nabla u_{tt}|^2dxdt\leq C(\tau),
\ea\ee
which, together with \eqref{cxb25} and \eqref{zjh3}, yields that
\be \la{zj37}
\sup_{\tau\le t\le T}\|\nabla u_t\|_{H^1}+\int_\tau^T\int|\nabla u_{tt}|^2dxdt\leq C(\tau).
 \ee
Now, \eqref{cxb58} follows from \eqref{cxb53}, \eqref{zj37}, and \eqref{cxb44}. We finish the proof of Lemma \ref{xle5}.
\end{proof}

\section{\la{se6}Proof of  Theorem  \ref{th1}}

Now that all the a priori estimates what we need have been obtained, we will prove the main results of this paper.

\begin{proof} By Lemma \ref{loc1}, there exists a
$T_*>0$ such that the  system (\ref{a1})-(\ref{ch1}) has a unique classical solution $(\rho,u)$ on $\Omega\times
(0,T_*]$. Now we use the a priori estimates, Proposition \ref{pr1} and Lemmas \ref{xle3}-\ref{xle5} to extend the local
solution $(\rho,u)$ to be a global one.

First, it is easy to check that
$$ A_1(0)+A_2(0)=0, \,\, 0\leq\rho_0\leq \hat{\rho},\,\, A_3(0)\leq M.$$
Therefore, there exists a
$T_1\in(0,T_*]$ such that
\be\la{dlbh1}\ba
0\leq\rho\leq2\hat{\rho},\,\,A_1(T)+A_2(T)\leq 2C_0^{1/3}, \,\, A_3(\sigma(T))\leq 3K
\ea\ee
holds for $T=T_1.$

Next, we set
\bn \la{dlbh2}
T^*=\sup\{T\,|\,{\rm (\ref{dlbh1}) \ holds}\}.
\en
Then $T^*\geq T_1>0$. Hence, for any $0<\tau<T\leq T^*$
with $T$ finite, it follows from Lemmas \ref{xle3}-\ref{xle5}
that
 \be \la{dlbh3}\ba
 \nabla u_t,\,\,\nabla^3u \in C([\tau ,T]; L^4),\quad
 \na u,\na^2u \in C\left([\tau ,T];
 C (\bar{\Omega})\right),\ea\ee
 where one has taken advantage of  the standard
embedding
$$L^\infty(\tau ,T;H^1)\cap H^1(\tau ,T;H^{-1})\hookrightarrow
C\left([\tau ,T];L^q\right),\quad\mbox{ for any } q\in [1,6).  $$
Due to (\ref{cxb17}), (\ref{cxb25}), (\ref{cxb58}) and $(\ref{a1})_1$,
we obtain
\be\ba
&\int_{\tau}^T \left|\left(\int\n|u_t|^2dx\right)_t\right|dt\\
&\le\int_{\tau}^T\left(\|  \n_t  |u_t|^2 \|_{L^1}+2\|  \n  u_t\cdot u_{tt} \|_{L^1}\right)dt\\
&\le C\int_{\tau}^T \left( \| \n|\div u||u_t|^2 \|_{L^1}+\|  |u||\na \n| |u_t|^2 \|_{L^1}+ \|\rho^{\frac{1}{2}}  u_t
\|_{L^2}\|\rho^{\frac{1}{2}}u_{tt} \|_{L^2}\right)dt\\
&\le C\int_{\tau}^T\left( \| \n |u_t|^2 \|_{L^1}\|\na u\|_{L^\infty}+\|  u\|_{L^6}\|\na\n\|_{L^2} \|u_t  \|^2_{L^6}+  \|\rho^{\frac{1}{2}}u_{tt} \|_{L^2}\right)dt\\
&\le C,\ea\ee
which together with \eqref{dlbh3} yields
\be\la{dlbh4} \rho^{1/2}u_t, \quad  \nabla\dot{u}\in C([\tau,T],L^2).\ee
Finally, we claim that \be \la{dlbh5}T^*=\infty.\ee Otherwise,
$T^*<\infty$. By Proposition \ref{pr1}, it holds that
\be\la{dlbh6}\ba
0\leq\rho\leq\frac{7}{4}\hat{\rho} ,\,\,\,A_1(T^*)+A_2(T^*)\leq C_0^{1/3},\,\,\, A_3(\sigma(T^*))\leq 2K .
\ea\ee
It follows from Lemmas \ref{xle3}, \ref{xle4}, \ref{xle5} and
(\ref{dlbh4}) that $(\n(x,T^*),u(x,T^*))$ satisfies
the initial data condition (\ref{dt1})-(\ref{dt3}),
where  $g(x)\triangleq\n^{1/2}\dot u(x, T^*),\,\,x\in \Omega.$ Thus, Lemma
\ref{loc1} implies that there exists some $T^{**}>T^*$ such that
(\ref{dlbh1}) holds for $T=T^{**}$, which contradicts \eqref{dlbh2}.
As a result, $0<T_1<T^*=\infty$. By Lemmas \ref{loc1} and \ref{xle3}-\ref{xle5}, it indicates that $(\rho,u)$ is in fact the unique globally classical solution.

It remains to prove \eqref{qa1w}. Integrating $\eqref{a1}_1$ over $\O\times (0,T)$ and using \eqref{ch1} yields that \be \la{bz11}\bar\n=\frac{1}{|\O|}\int\n (x,t)dx\equiv \frac{1}{|\O|} \int \n_0dx. \ee

For $G(\rho)$, there exists a suitably small positive constant $\tilde{C} <1$ depending only on $a,\,\gamma,\,\bar{\rho}_0,$ and $\hat \n$ such that for any $\rho\in [0,2\hat\n]$,
\be\label{gine1}  \tilde{C}^2( \rho-\bar{\rho})^2\le \tilde{C}  G(\rho)  \leq    (\rho^\gamma-\bar{\rho}^\gamma)( \rho - \bar{\rho}). \ee
Consider the problem
\bn\la{e4800}\begin{cases}
{\rm div}\phi=\rho-\bar{\rho}, \,\,\,\,  &x\in\Omega, \\
\phi=0,\,\,\,&x\in{\partial\Omega},
\end{cases} \en
where $\Omega$ is a bounded domain in $R^{3}$ with smooth boundary.

Multiplying $\eqref{a1}_2$ by $\phi$, we get
\be\la{c59} \ba &
\int(P-P(\bar\n))(\n-\bar\n) dx \\&= \left(\int\rho u\cdot\phi dx\right)_t-\int\rho u\cdot\nabla\phi\cdot udx - \int\rho u\cdot\phi_t  dx \\
& \quad  +\mu\int\nabla u\cdot\nabla\phi dx +\int(\lambda(\rho)+\mu)(\n-\bar\n)\div udx \\
& \leq \left(\int\rho u\cdot\phi dx\right)_t+C(\hat{\rho})\|u\|_{L^{4}}^{2}\|\n-\bar\n\|_{L^2} +C(\hat{\rho})\|\rho u\|_{L^2}^2\\
& \quad  +C\|\rho-\bar{\rho}\|_{L^2}\|\nabla u\|_{L^2} \\
& \leq \left(\int\rho u\cdot\phi dx\right)_t+\de \|\n-\bar\n\|_{L^2}^2 +C(\de)\|\na u\|_{L^2}^2,
\ea\ee
which, along with \eqref{gine1} and \eqref{tdu1}, leads to
\be \la{gine2} \ba
a\tilde{C}\int G(\rho)dx&\leq a\int(\rho^\gamma-\bar{\rho}^\gamma)( \rho - \bar{\rho})dx\\
&\leq 2\left(\int\rho u\cdot\phi dx\right)_t+C_1H(t),
\ea\ee
where $H(t)\triangleq \int(\lambda+2\mu)(\div u)^{2}dx+\mu\int|\curl u|^{2}dx.$

Moreover, it follows from \eqref{gine1} and Young's inequality that
\be \la{c511} \ba
\left|\int\rho u\cdot\phi dx\right|\leq C_2\left(\frac{1}{2}\|\sqrt{\rho} u\|^2_{L^2}+\int G(\rho)dx\right),
\ea\ee
 which gives
\be \la{c512} \ba
\frac{1}{2}\left(\frac{1}{2}\|\sqrt{\rho} u\|^2_{L^2}+\int G(\rho)dx\right)\leq W(t)\leq \frac{3}{2}\left(\frac{1}{2}\|\sqrt{\rho} u\|^2_{L^2}+\int G(\rho)dx\right),\ea\ee
 where $$W(t)=\int \left(\frac{1}{2}\rho |u|^2+G(\rho)\right)dx-\delta_0\int\rho u\cdot\phi dx,$$ with $\delta_0=\min\{\frac{1}{2C_1},\frac{1}{2C_2}\}.$

Adding \eqref{gine2} multiplied by $\de_0 $ to \eqref{buchong}
and using \bnn \int\n |u|^2dx\leq C(\hat{\rho})\|\na u\|_{L^2}^2\leq C_3H(t),\enn we obtain for a suitably small constant $\delta_1=\delta_1(a,\delta_0, \tilde{C}_0,C_3)$,
 \bnn W'(t)+2\delta_1W(t)\leq 0,\enn which together with \eqref{c512} yields that for any $t\geq0$,
\begin{equation}\label{c513}
\int\left(\frac{1}{2}\rho|u|^2+G(\rho)\right)dx\leq 4C_0e^{-2\delta_1t}.
\end{equation}
By \eqref{buchong}, we have
\be \la{c514} \ba
\int_0^\infty H(t)e^{\delta_1 t} dt\leq C.
\ea\ee

Choosing $m=0$ in \eqref{I0}, by \eqref{I1}, \eqref{I2} and \eqref{I3}, we obtain
\be\la{c515}\ba
&\left(H(t)-2\int(P-P(\bar{\rho}))\,\div udx\right)_{t}+\frac{1}{2}\|\sqrt{\rho}\dot{u}\|^2_{L^2}\\
&\leq C\left(H(t)+\int G(\rho)dx\right).
\ea \ee
Notice that $$\|P-\bar{P}\|_{L^2}^2\leq C\|P-P(\bar{\rho})\|_{L^2}^2\leq C\int G(\rho)dx.$$
Multiplying \eqref{c515} by $e^{\delta_1 t}$, and using the fact
\bnn\ba
\left|\int(P-P(\bar{\rho})\,\div udx\right|\leq C\int G(\rho)dx+\frac{1}{4}H(t),
\ea\enn
we get
\be\la{c516}\ba
&\left(e^{\delta_1 t}H(t)-2e^{\delta_1 t}\int(P-P(\bar{\rho})\,\div udx\right)_{t}+\frac{1}{2}e^{\delta_1 t}\|\sqrt{\rho}\dot{u}\|^2_{L^2}\\
&\leq Ce^{\delta_1 t}\left(H(t)+\int G(\rho)dx\right),
\ea \ee
which, together with \eqref{c513} and \eqref{c514}, yields that for any $t>0$,
\be\la{c517}\ba
\|\nabla u\|_{L^2}^2\leq Ce^{-\delta_1 t},
\ea \ee
and
\be\la{c518}\ba
\int_0^\infty e^{\delta_1 t}\|\sqrt{\rho}\dot{u}\|^2_{L^2}dt\leq C.
\ea \ee
By \eqref{ax401}, \eqref{h99}, \eqref{c517} and \eqref{c518}, a direct calculation leads to
\be\la{c519}\ba
\|\sqrt{\rho}\dot{u}\|^2_{L^2}\leq Ce^{-\delta_1 t},\,\,t\geq1.
\ea \ee

Finally, together with \eqref{c513}, \eqref{c517}, \eqref{c519} and \eqref{h18}, we obtain \eqref{qa1w} for some positive constant $\tilde{\eta}\leq\delta_1$ depending only on $\mu,$  $\lambda,$  $\gamma,$ $a$, $\on$, $\hat{\rho}$, $M$, $\Omega$, $p$, $r$ and $C_0$ and  finish the proof.
\end{proof}
\begin{proof}
Suppose $T>0$, we introduce the Lagrangian coordinates
  \be \la{c61}  \begin{cases}\frac{\partial}{\partial \tau}X(\tau; t,x) =u(X(\tau; t,x),\tau),\,\,\,\, &0\leq \tau\leq T,\\
 X(t;t,x)=x, \,\,\,\, &0\leq t\leq T,\,x\in\bar{\Omega}.\end{cases}\ee
By \eqref{dt6}, it is easy to find that \eqref{c61} is well-defined. \eqref{c61} together with $\eqref{a1}_1$ shows
 \be\la{c62}\ba
\rho(x,t)=\rho_0(X(0; t, x)) \exp \left\{-\int_0^t\div u(X(\tau;t, x),\tau)d\tau\right\}.
\ea \ee

If there exists some point $x_0\in \Omega$ such that $\n_0(x_0)=0,$ then for any $t>0$, $X(0; t, x_0(t))=x_0$. Hence, for any $t\geq 0,$ $\rho(x_0(t),t)\equiv 0$ due to \eqref{c62}. As a result, Gagliardo-Nirenberg's inequality shows that for any $\tilde{q}\in(1,\infty)$ and $ \tilde{r}\in  (3,\infty),$
\be\la{c63}\ba\bar{\rho}_0\equiv\bar\n\leq\|\rho-\bar{\rho}\|_{C\left(\ol{\O }\right)} \le C
\|\rho-\bar{\rho}\|_{L^{\tilde{q}}}^{\tilde{\theta}}\|\na \rho\|_{L^{\tilde{r}}}^{1-\tilde{\theta}}
\ea\ee
where $\tilde{\theta}=\tilde{q}(\tilde{r}-3)/(3\tilde{r}+\tilde{q}(\tilde{r}-3))$. Together with \eqref{qa1w}, we gives \eqref{qa2w}. This completes the proof.
\end{proof}


\begin{thebibliography}{99}
\bibitem{CANEHS}
Aramaki, J.: $L^p$ theory of the div-curl system, Int. J. Math. Anal. \textbf{8}(6)(2014), 259-271.
\bibitem{Bresch1}
Bresch, D., Desjardins, B.: Existence of global weak solutions for a 2D viscous shallow water equations and convergence
to the quasi-geostrophic model. Commun. Math. Phys. \textbf{238}(1-2)(2003), 11-223.
\bibitem{Bresch2}
Bresch, D., Desjardins, B., Lin, C. K.: On some compressible fluid models: Korteweg, lubrication, and shallow water
systems, Comm. Partial Differ. Equ. \textbf{28}(3-4)(2003), 843-868.
\bibitem{CCL1}
Cai G. C., Li J.: Existence and Exponential Growth of Global Classical Solutions to the Compressible Navier-Stokes Equations with Slip Boundary Conditions in 3D Bounded Domains. https://arxiv.org/pdf/2102.06348.pdf
\bibitem{Cpfc1}
Constantin, P., Foias, C.: Navier-Stokes equations. Chicago Lectures in Mathematics. University of Chicago Press, Chicago, 1988.
\bibitem{Ding}
Ding, S. J., Wen, H. Y., Zhu, C. J.: Global classical large solutions to 1D compressible Navier-Stokes equations with
density-dependent viscosity and vacuum. J. Differ. Equ. \textbf{251}(2011), 1696-1725.
\bibitem{GPG}
Galdi, G. P.: An introduction to the mathematical theory of the Navier-Stokes equations, I(Springer, New-York, 1994).
\bibitem{H3}
Hoff, D.: Global solutions of the Navier-Stokes equations for multidimensional compressible flow with discontinuous initial data.
\textit{J. Differ. Eqs.}   \textbf{120}(1)(1995), 215-254.
\bibitem{Ho3}
Hoff, D.:  Compressible flow in a half-space with Navier boundary conditions. \textit{J. Math. Fluid Mech.}  \textbf{7}(3)(2005), 315-338.
\bibitem{hxd1} Huang, X. D. On local strong and classical solutions to the three-dimensional barotropic compressible
Navier-Stokes equations with vacuum. Sci China Math, 2020, 63, https://doi.org/10.1007/s11425-019-
9755-3
\bibitem{hlx}
Huang, X. D., Li, J.; Xin Z. P.: Serrin type criterion for the three-dimensional compressible flows. \textit{SIAM J. Math. Anal.},   \textbf{43}(4)(2011), 1872-1886.
\bibitem{hlx1}
Huang, X. D., Li, J., Xin, Z. P.:    Global well-posedness of classical solutions with large oscillations and vacuum to the three-dimensional isentropic
compressible Navier-Stokes equations.{\it Comm. Pure Appl. Math.}  {\bf65} (2012), 549-585.
\bibitem{itaya}
Itaya, N.: On the Cauchy problem for the system of fundamental equations describing the movement of compressible
viscous fluids. Kodia Math. Sem. Rep. \textbf{23}(1971), 60-120.
\bibitem{Itt2}
Itoh, S., Tanaka, N., Tani, A.: The initial value problem for the Navier-Stokes equations with general slip boundary condition in H\"{o}lder spaces. J. Math. Fluid Mech. \textbf{5}(3)(2003),  275-301.
\bibitem{JXZ}
Jiang, S., Xin, Z. P., Zhang, P.: Global weak solutions to 1D compressible isentropy Navier-Stokes with density-dependent
viscosity. Methods Appl. Anal. \textbf{12}(3)(2005), 239-252.
\bibitem{liliang}
Li, J., Liang, Z.L.: On classical solutions to the cauchy problem of the two-dimensional barotropic compressible Navier-
Stokes equations with vacuum. J. Math. Pure Appl. \textbf{102}(4)(2014), 640-671.
\bibitem{lx}
Li, J., Xin, Z. P:  Some uniform estimates and blowup behavior of global strong solutions to the Stokes approximation equations for two-dimensional compressible flows. \textit{J. Differ. Eqs.}   \textbf{221}(2)(2006), 275-308.
\bibitem{jx01}
Li, J., Xin, Z. P:  Global Existence of Regular Solutions with Large Oscillations and Vacuum. In: Giga Y., Novotny A. (eds) Handbook of Mathematical Analysis in Mechanics of Viscous Fluids. Springer 2016.
\bibitem{lixinzhou}
Li, J., Xin, Z. P.: Global well-posedness and large time asymptotic behavior of classical solutions to the compressible
Navier-Stokes equations with vacuum. http://arxiv.org/abs/1310.1673
\bibitem{LXY}
Liu, T. P., Xin, Z. P., Yang, T.: Vacuum states of compressible flow. Discrete Contin. Dyn. Syst. \textbf{4}(1998), 1-32.
\bibitem{M1}
Matsumura, A.,  Nishida, T.:   The initial value problem for the equations of motion of viscous and heat-conductive gases. \textit{J. Math. Kyoto Univ.}  \textbf{20}(1980), no. 1, 67-104.
\bibitem{Na}
Nash, J.:  Le probl\`{e}me de Cauchy pour les \'{e}quations diff\'{e}rentielles d'un fluide g\'{e}n\'{e}ral. \textit{Bull. Soc. Math. France.}  \textbf{90} (1962), 487-497.
\bibitem{Nclm1}
Navier, C. L. M. H.: Sur les lois de l'\'{e}quilibre et du mouvement des corps \'{e}lastiques. Mem. Acad. R. Sci. Inst. France 6 (1827), 369.
\bibitem{nir}
 Nirenberg, L.:  On elliptic partial differential equations. {\it Ann. Scuola Norm. Sup. Pisa}   {\bf 13}(1959), 115-162.
\bibitem{se2}
Serrin, J.: Mathematical principles of classical fluid mechanics. Handbuch der Physik (herausgegeben von S. Fl\"{u}gge), Bd. 8/1, Str\"{o}mungsmechanik I (Mitherausgeber C. Truesdell), 125-263. Springer, Berlin-G\"{o}ttingen-Heidelberg, 1959.
\bibitem{Sva1}
Solonnikov, V. A., \v{S}\v{c}adilov, V. E.: A certain boundary value problem for the stationary system of Navier-Stokes equations. Boundary value problems of mathematical physics, 8. Trudy Mat. Inst. Steklov. 125 (1973), 196-210. Translation in Proc. Steklov Inst. Math. \textbf{125} (1973), 186-199.
\bibitem{TA1}
Tani, A.: On the first initial-boundary value problem of compressible viscous fluid motion, Publ.RIMS Kyoto Univ. \textbf{13}(1977), 193-253.
\bibitem{Vka1}
Vaigant, V. A., Kazhikhov, A.V.: On the existence of global solutions of two-dimensional Navier-Stokes equations of
a compressible viscous fluid (Russian). Sibirsk. Mat. Zh. 36(6), 1283-1316 (1995); translation in Sib. Math. J. \textbf{36}(6)(1995),
1108-1141.
\bibitem{vww}
von Wahl, W.: Estimating $\nabla u$ by $\div u$ and $\curl u$. Math. Methods in Applied Sciences, \textbf{15}(1992), 123-143.
\bibitem{YYZ}
Yang, T., Yao, Z.A., Zhu, C. J.: Compressible Navier-Stokes equations with density-dependent viscosity and vacuum.
Commun. Partial Differ. Equ. \textbf{26}(2001), 965-981.
\bibitem{ZF}
Zhang, T., Fang, D. Y.: Global behavior of spherically symmetric Navier-Stokes-Poisson system with degenerate viscosity
coefficients, Arch. Ration. Mech. Anal. \textbf{191}(2009), 195-243.
\bibitem{zl1}
Zlotnik, A. A.:    Uniform estimates and stabilization of symmetric solutions of a system of quasilinear equations.  \textit{Diff. Eqs.} \textbf{36} (2000),  701-716.
 \end{thebibliography}
\end{document}